\documentclass[a4paper, 11pt]{article}
\usepackage{amsmath,epsfig, calc, amsthm}
\usepackage{amssymb, color, mathrsfs}
\usepackage{bbm, url}
\usepackage{enumitem} 
\usepackage[numbers]{natbib}
\usepackage{todonotes}

\usepackage{hyperref} 
\usepackage{mathtools}

\newtheorem{theorem}{Theorem}[section]

\newtheorem{proposition}[theorem]{Proposition}

\newtheorem{lemma}[theorem]{Lemma}
\newtheorem*{theorem*}{Theorem}

\theoremstyle{definition}
\newtheorem{definition}[theorem]{Definition}
\newtheorem{remark}[theorem]{Remark}
\definecolor{javared}{rgb}{0.6,0,0}
\definecolor{javagreen}{rgb}{0.25,0.5,0.35}
\definecolor{javapurple}{rgb}{0.5,0,0.35}
\definecolor{javadocblue}{rgb}{0.25,0.35,0.75}

\newcommand{\1}{\mathbbm{1}}

\newcommand{\cB}{\mathcal{B}}
\newcommand{\cC}{\mathcal{C}}

\newcommand{\cE}{\mathcal{E}}

\newcommand{\cG}{\mathcal{G}}

\newcommand{\cS}{\mathcal{S}}
\newcommand{\cO}{\mathcal{O}}

\newcommand{\de}{\operatorname{d}}
\newcommand{\ra}{\rightarrow}

\newcommand{\sC}{\mathscr{C}}

\DeclareMathOperator{\E}{\mathbb{E}}
\DeclareMathOperator{\p}{\mathbb{P}}
\DeclareMathOperator{\diam}{diam}
\newcommand{\Om}{\Omega}

\newcommand{\comps}{{\rm Comp}(G_N)}

\begin{document}

\begin{center}
{\Large \bf Voter models on subcritical inhomogeneous\\[1mm] random graphs}\\
\vspace{0.7cm}
\textsc{John Fernley\footnote{Department of Mathematical Sciences, University of Bath, Claverton Down, Bath, BA2 7AY,
United Kingdom, {\tt j.d.fernley@bath.ac.uk}, {\tt m.ortgiese@bath.ac.uk}.}}
and \textsc{Marcel Ortgiese\footnotemark[1]}\\ 
\end{center}

\begin{abstract}
The voter model is a classical interacting particle system modelling how consensus is formed across a network. 
We  analyse the time to consensus for the voter model when the underlying graph is a subcritical scale-free random graph. Moreover, we generalise the model to include a `temperature' parameter. The interplay between the temperature and the structure of the random graph leads to a very 
rich phase diagram, where in the different phases different parts of the underlying geometry dominate the time to consensus. 
Finally, we also consider a discursive voter model, where voters discuss their opinions with their neighbours. Our proofs rely on the well-known duality to 
coalescing random walks and a detailed understanding of the structure of the random graphs.

  

\end{abstract}


\section{Introduction}

Voter models are a classical example of an interacting particle system defined on a graph that models how consensus is formed e.g.\ across a social network.
In the standard model voters can have one of two opinions and at rate $1$ a vertex updates and copies the opinion from one of its neighbours.

Classically, the underlying graph is $\mathbb{Z}^d$, see e.g.~\cite{liggett1985interacting}, and typical questions study the structure and existence of invariant measures.
When considered on a finite graph, the invariant measures become trivial and the main question is how long it takes to reach consensus.
One example when this question can be answered is the complete graph on $N$ vertices, in which case the voter model is a variation of the Moran model from evolutionary biology and it is known that consensus is reached in time of order $N$. 

More recently, the voter model has also been studied on random graphs, see e.g~\cite[Chapter 6.9]{durrett2007random}, and any result is very much dependent on the underlying model.
Moreover, the case when the underlying graph is inhomogeneous in the sense that its empirical degree distribution shows power law behaviour has not
been treated systematically. In the nonrigorous literature, this analysis has been carried out by~\cite{sood2008voter} for a mean-field model.

It is well known that voter models are dual to coalescing random walks (which in some sense trace back where opinions came from). In particular, 
if we consider the final coalescence time at which a set of random walks started at each vertex in the underlying graph have coalesced into a single walker, then 
any paper that bounds the final coalescence time  also bounds the consensus time in the voter model. Examples
include \cite{oliveira2013mean, kanade2016coalescence}.

In this paper, we consider the voter model on subcritical inhomogeneous random graphs showing power law behaviour. The reason that we focus 
on subcritcal random graphs is that, as we will see below, the behaviour observed here cannot be captured by mean-field methods. Moreover, in our model the 
random graphs are disconnected, but the components are still of polynomial order in the number of vertices and have the fractal-like structure seen in Figure~\ref{component_figure}. Therefore,  the asymptotics of the consensus time (defined as the first time that all components reach consensus) depends on a subtle interplay of the different structures of the components.

Furthermore, we also introduce a ``temperature'' parameter $\theta \in \mathbb{R}$ into the model, so that vertices update at rate proportional to  $\de(v)^\theta$, where $\de(v)$ denotes the degree of vertex $v$. This extra parameter leads to  interesting phase transitions in $\theta$, where in the different phases different structural elements of the underlying random graphs dominate the consensus time.

Finally, we also consider a variation of the voter model, also considered by \cite{moinet2018generalized} and similar to the ``oblivious'' model of \cite{cooper2018discordant}, which we will refer to as the \emph{discursive voter model}. In this model, vertices update at rate $\de(v)^\theta$, but then they `discuss' with a randomly chosen neighbour and agree on one opinion chosen at random from their respective opinions. This model gives a very different phase diagram -- surprisingly, the large components do not explain the consensus time for large $\theta$ or for small $\theta$.

Our proofs  rely on duality to  coalescing random walkers and using the right tools to bound the expected coalescence times. This is combined with a very fine analysis of the structure of subcritical inhomogeneous random graphs models that is not readily available in the literature.

\begin{figure}[t!]
\centering
\vspace{-1cm}
\centerline{\includegraphics[width=1\textwidth ]{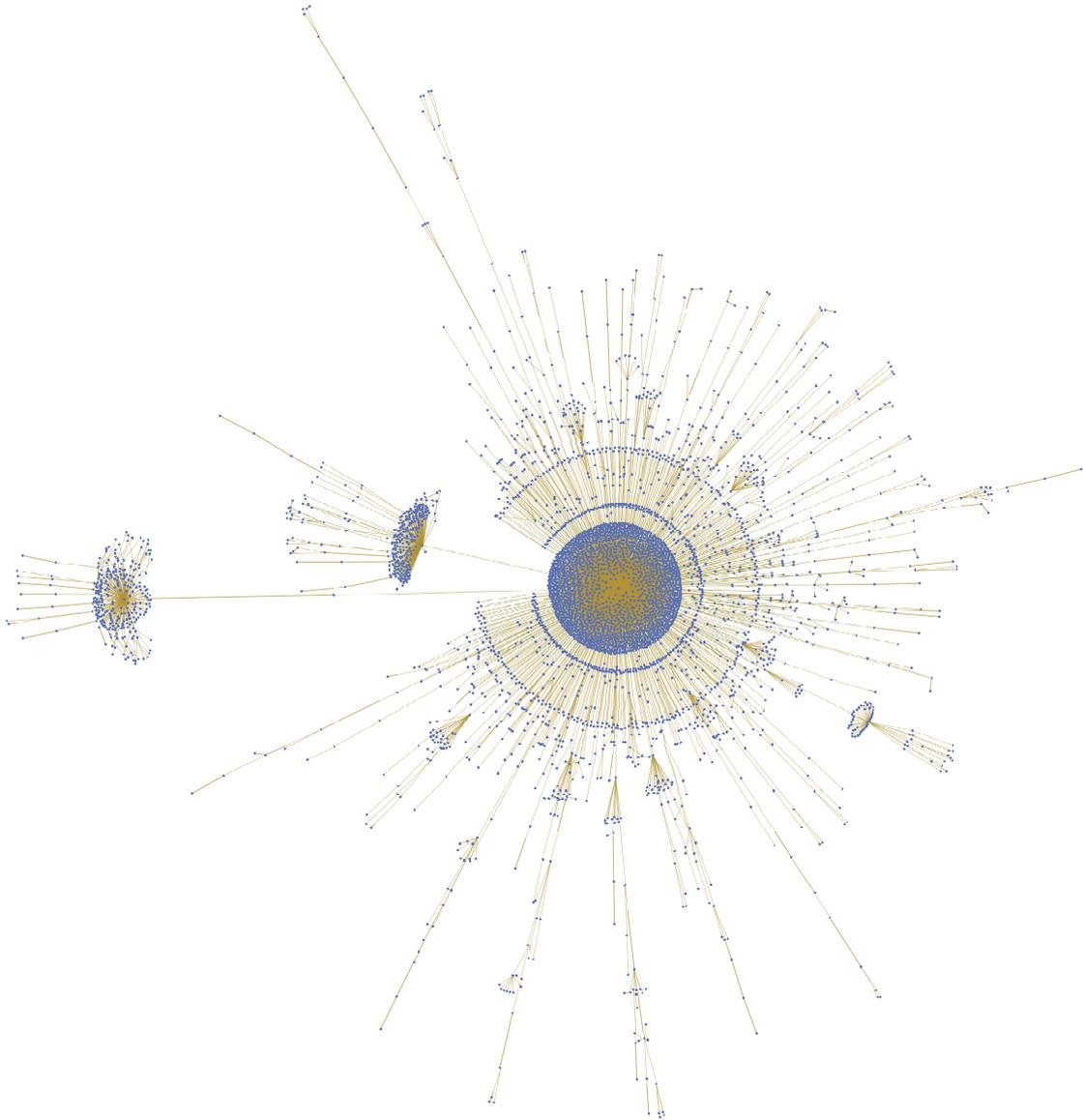}}
\caption{The component containing the vertex $1$, for a graph in the class $\cG_{\beta,\gamma}$ with subcritical network parameters $(\beta,\gamma)=(0.05,0.45)$. On these $4429$ vertices we can already see the emerging fractal structure.}\label{component_figure}
\end{figure}

{\bf Notation.} Throughout the paper, we will use the following notation.
For sequences of positive random variables $(X_N)_{N \geq 1}$ and $(Y_N)_{N \geq 1}$, we write
 $X_N=O_{\mathbb{P}}^{\log N}(Y_N)$ if there exists $K> 0$ such that 
\[
\mathbb{P}\left(X_N\leq Y_N (\log N)^K \right)\rightarrow 1
\]
as $N \rightarrow \infty$. Similarly, we write  $X_N=\Omega_{\mathbb{P}}^{\log N}(Y_N)$ if
 $Y_N=O_{\mathbb{P}}^{\log N}(X_N)$. If both bounds hold we write $X_N=\Theta_{\mathbb{P}}^{\log N}(Y_N)$.
 
Throughout we write $[N] = \{1,\ldots, N\}$.
For any graph $G$, we write $V(G)$ for its vertex set and $E(G)$ for its edge set. Moreover, if $v , w\in V(G)$, we write $v \sim w$  if $v$ and $w$ are neighbours, i.e.\ if $\{ v, w \} \in E(G)$. Also, for $v \in V(G)$, we
denote by $\de(v)$ its degree (i.e.\ the number of its neighbours).

\section{Main results}

In this paper, we will consider the following two variants of the voter model.

\begin{definition}\label{def:voter}
Let $G = (V,E)$ be a (simple) finite graph. Given $\eta \in \{ 0,1\}^V$, define for $i\neq j \in V$, 
\[ \eta^{i \leftarrow j} (k) = \left\{ \begin{array}{ll} \eta(j)  & \mbox{if } k = i \in V, \\ \eta(k)  & \mbox{if } k \in V \setminus \{ i \} . \end{array}\right. \]
The \emph{voter model} is a Markov process $(\eta_t)_{t \geq 0}$ with state space $\{0,1\}^V$ and with the following update rules depending on a parameter $\theta \in \mathbb{R}$:
\begin{itemize}
	\item[(a)]  In the \emph{classical} voter model, for any neighbours $i$ and $j \in V$, the state $\eta \in \{0,1\}^V$ is replaced by $\eta^{i \leftarrow j}$ at rate
	\[ \de(i)^{\theta -1} . \]
	\item[(b)] For the \emph{discursive voter model}, for any neighbours $i$ and $j \in V$, the state $\eta \in \{0,1\}^V$ is replaced by $\eta^{i \leftarrow j}$ at rate
	\[ \frac{1}{2} \big(\de(i)^{\theta -1} + \de(j)^{\theta -1}\big) . \]
\end{itemize}
\end{definition}

The classical voter model has the interpretation that each vertex $i$ updates its opinion at rate $\de(i)^\theta$ by copying the opinion of a uniformly chosen neighbour.
In the discursive model, each vertex $i$ becomes active at rate $\de(i)^\theta$, then it chooses  a neighbour uniformly at random and then both neighbours agree on a common opinion by picking one of their opinions randomly.
Note also that for the `temperature' parameter $\theta$, $\theta = 0$ corresponds to the `standard' voter model, where each vertex $i \in [N]$ updates its opinion at rate $1$.

We consider a general class of inhomogeneous random graphs, which include a special case of the Chung-Lu model.  The latter has vertex set $[N]$ and each edge is present independently  with probability
\begin{equation}\label{chung_lu_edges}
p_{ij}:=\frac{\beta N^{2\gamma-1}}{i^\gamma j^\gamma}\wedge 1.
\end{equation}
We generalize this model as follows.

\begin{definition}\label{define_G} Fix $\beta, \gamma >0$ such that $\beta+2\gamma<1$. 
We say that a sequence of (simple) random graphs $(G_N)_{N \geq 1}$, where $V(G_N) = [N]$, is 
in the class $\cG_{\beta,\gamma}$ of \emph{subcritical inhomogeneous random graph models}
with parameters $\beta$ and $\gamma$, if for any $N$ there exists a symmetric array $(q_{i,j})_{i ,j \in [N]}$ of numbers in $(0,\frac 12)$ such that each edge $\{i,j\}$, $i \neq j$, is present in $G_N$ independently of all others with probability $q_{i,j}$. Moreover, for $(p_{i,j})_{i,j \in [N]}$  as in~\eqref{chung_lu_edges},  we require that
\[
\lim_{N\rightarrow\infty} \sum_{i \neq j} \frac{(p_{i,j}-q_{i,j})^2}{p_{i,j}}=0.
\]
\end{definition}

\begin{remark}\label{rem:IRG}
\begin{itemize} 
\item[(a)] Our definition includes various well-known models of inhomogeneous random graphs additionally to the Chung-Lu (CL) model, where $q_{i,j} = p_{i,j}$.
Other models are the simple Norros-Reittu random graph (SNR) with $q_{i,j}=1-e^{-p_{i,j}}$ obtained from the Norros-Reittu model by flattening all multiedges, see also Section~\ref{Branching_section} below. 
It also includes  the Generalised Random Graph (GRG) with $q_{i,j}=\frac{p_{i,j}}{1+p_{i,j}}$, which has the distribution of a particular configuration model conditioned to be simple \cite[Theorem 7.18]{van2016random}.
\item[(b)] In the following, we will slightly abuse notation and write $G_N \in \mathcal{G}_{\beta, \gamma}$ if we mean that $G_N$ is the random graph with vertex set $[N]$ in a sequence of random graphs in $\mathcal{G}_{\beta, \gamma}$.
\item[(c)]
Any two representatives of the class $\mathcal{G}_{\beta, \gamma}$ are asymptotically equivalent, see~\cite[Theorem 6.18]{van2016random}, so that if a statement holds with high probability  for one particular model  it also holds for any other one from the class. In the proofs,  we will sometimes make use of this freedom and choose a particular model when it is convenient.
\item[(d)] By \cite{bollobas2007phase} we know that the regime $\beta +2\gamma<1$ is the complete subcritical region, so in particular our class of random graphs does not have a giant component.
Also, by~\cite{chung2006complex} it is known that these models have power-law exponent $\tau = 1 + 1/\gamma$.
\end{itemize}
\end{remark}

Note that our network model is typically disconnected, so then the voter model $(\eta_t)_{t \geq 0}$ can hit an absorbing state without being in global consensus.
So, let $C_1, \ldots, C_k$ be the components of a graph $G$ on the vertex set $[N]$.
The \emph{consensus time} is the first time that there is consensus on each component, i.e.\ 
\[ \tau_{\rm cons} = \inf\{ t \geq 0 \, : \, \eta_t|_{C_i} \mbox{ is constant for each } i \in [k] \} . \]

We can now state our first main theorem on the expected consensus time.

\begin{theorem}\label{class_subcrit}
Suppose  $G_N \in \cG_{\beta,\gamma}$ for some $\beta+2\gamma<1$ and that the initial conditions are chosen according to $\mu_u$ such that each initial opinion is an  independent $\operatorname{Bernoulli}(u)$ random variable, for some $u \in (0,1)$. Then, for the classical voter model with parameter $\theta \in \mathbb{R}$, we have
\begin{equation}\label{eq:2611-1}
\mathbb{E}^{\theta}_{\mu_u}(\tau_{\text{\textnormal{cons}}}|G_N)=
\Theta_{\mathbb{P}}^{\log N}\left( N^{\mathbbm{c}}\right)
\end{equation}
where the exponent $\mathbbm{c}=\mathbbm{c}(\gamma,\theta)$ is given as
\[
\mathbbm{c}=
\begin{cases}
\gamma & \theta\geq 1 ,\\
\gamma \theta & \frac{1}{2-2\gamma}<\theta< 1 , \\
\frac{\gamma}{2-2\gamma} & 0 \leq \theta\leq \frac{1}{2-2\gamma} ,\\
\frac{\gamma(1- \theta)}{2-2\gamma} & \theta < 0.
\end{cases}
\]
\end{theorem}

Note that  we are only averaging over the voter model dynamics, so that the expectation in~\eqref{eq:2611-1} is still random (but depends only on the realization of the particular random graph). 

The averaged $\theta=0$ classical dynamics are studied in~\cite{sood2008voter} and they find a linear expected time to hit \emph{global} consensus whenever $\gamma<1/2$. In our non-averaged dynamics however the consensus is  componentwise and hence  we find a sublinear consensus time $N^\frac{\gamma}{2-2\gamma}=N^\frac{1}{2\tau-4}=o( \sqrt{N} )$.

\begin{remark}[Dominating contributions for the classical model]\label{class_dominating_contrib} We remark that the theorem shows that dominating contributions to the consensus time come from different parts of the random graph in the different regimes. 
The proof can be adapted to show that on $\sC(1)$, the component of vertex $1$, which  is with high probability the largest component, we have the following asymptotics: 
\[
\mathbb{E}^{\theta}_{\mu_u}(\tau_{\text{\textnormal{cons}}}(\sC(1))|G_N)
=
\Theta_{\mathbb{P}}^{\log N}\left(N^c\right), \quad \text{where }
c=
\begin{cases}
\gamma & \theta\geq 1 ,\\
\gamma \theta & \frac{\gamma}{1-\gamma}<\theta< 1 , \\
\frac{\gamma^2}{1-\gamma} & 0 \leq \theta\leq \frac{\gamma}{1-\gamma} ,\\
\frac{\gamma^2(1- \theta)}{1-\gamma} & \theta < 0.
\end{cases}
\]
Therefore, by comparison with Theorem \ref{class_subcrit}, in the regime $\theta<1/(2-2\gamma)$ we find $\sC(1)$ is not the component that takes longest to reach consensus. Instead in this case, the consensus time of Theorem \ref{class_subcrit} the dominating contribution comes from 
the consensus time on  a \emph{double star} component, i.e.\ a tree component with two connected vertices of polynomially maximal degree, which exists with high probability.
\end{remark}

Next we consider the consider the discursive model, where we have the following phase diagram.

\begin{theorem}\label{obl_subcrit}
Suppose  $G_N \in \cG_{\beta,\gamma}$ for some $\beta+2\gamma<1$ and that the initial conditions are chosen according to $\mu_u$ such that each initial opinion is an  independent $\operatorname{Bernoulli}(u)$ random variable, for some $u \in (0,1)$. Then, for the discursive voter model with parameter $\theta \in \mathbb{R}$, we have
\[
\mathbf{E}^{\theta}_{\mu_u}(\tau_{\text{\textnormal{cons}}}|G_N)=
\Theta^{\log N}_{\mathbb{P}}\left( N^{\mathbf{c}} \right)
\]
where the exponent $\mathbf{c}=\mathbf{c}(\gamma,\theta)$ is given as
\[
\mathbf{c}=
\begin{cases}
 \frac{\gamma}{2-2\gamma}  & \theta \geq \frac{3-4\gamma}{2-2\gamma} , \\
 \gamma(2-\theta) & 1 < \theta < \frac{3-4\gamma}{2-2\gamma} , \\
 \gamma & 2\gamma \leq \theta \leq 1 ,\\
 \frac{\gamma(2-\theta)}{2-2\gamma} & \theta < 2\gamma.
\end{cases}
\]
\end{theorem}

Unlike for the classical model, where large positive $\theta$ slowed down consensus when compared to the standard model $\theta=0$, for the discursive model we see that large $\theta$ accelerates consensus by accelerating mixing: for each $\gamma$, $\mathbf{c}(\gamma,\theta)$ is non-increasing in $\theta$. See also Figure~\ref{fig:exponent_figure} for an illustration.

\begin{figure}[b!]
\centering
\centerline{\includegraphics[width=1.2\textwidth ]{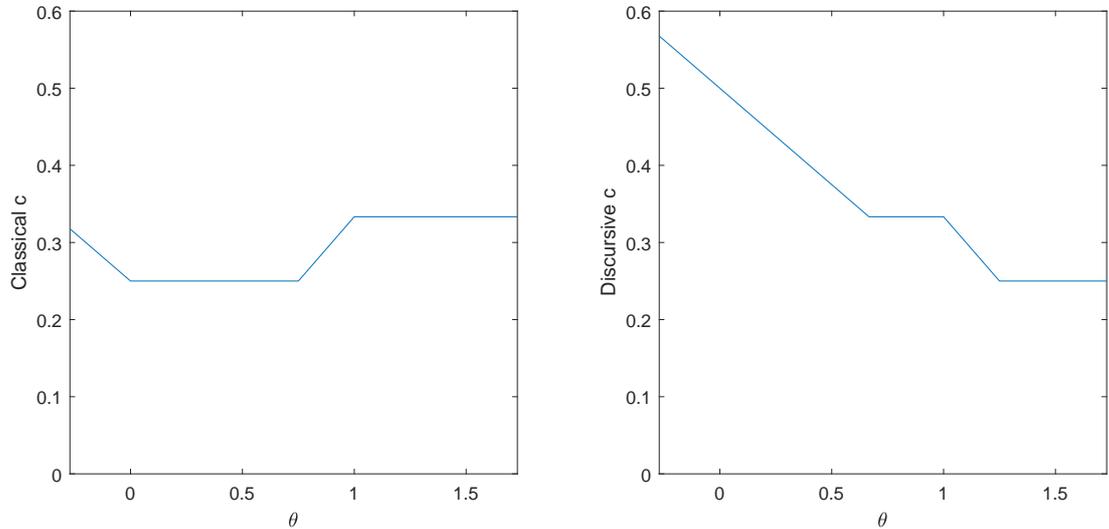}}
\caption{This figure shows the typical shapes by setting $\gamma=1/3$. Somewhat surprisingly, for any subcritical $(\beta,\gamma)$ parameters the function $\mathbbm{c}(\gamma,\theta)$ is \emph{not} monotonic in $\theta$ for the classical model. On the left we see that the most popular model $\theta=0$ is one of the fastest parameter values for consensus, as part of the optimal interval $\theta \in [ 0 , {1}/({2-2\gamma})]$. Conversely, the discursive model shows monotonicity in the exponent.}\label{fig:exponent_figure}
\end{figure}

\begin{remark}[Dominating contributions for the discursive model]
For comparison, we state the consensus order of $\sC(1)$ with the discursive dynamic.
\[
\mathbf{E}^{\theta}_{\mu_u}(\tau_{\text{\textnormal{cons}}}(\sC(1))|G_N)
=
\Theta_{\mathbb{P}}^{\log N}\left(N^c\right), \quad \text{where }
c=
\begin{cases}
\frac{\gamma^2}{1-\gamma} & \theta\geq \frac{2-3\gamma}{1-\gamma} ,\\
\gamma(2-\theta)  & 1<\theta< \frac{2-3\gamma}{1-\gamma} , \\
\gamma & 3-\frac{1}{\gamma} \leq \theta\leq 1 ,\\
\frac{\gamma^2(2- \theta)}{1-\gamma} & \theta < 3-\frac{1}{\gamma}.
\end{cases}
\]
The most obvious difference here is that $\sC(1)$ makes a dominating contribution to the consensus order on $G_N$, as seen in Theorem \ref{obl_subcrit}, only for parameters $\theta$ in a an \emph{intermediate} range $\theta \in [2\gamma, \frac{3-4\gamma}{2-2\gamma} ]$ as opposed to in Remark \ref{class_dominating_contrib} where this was true for $\theta$ sufficiently large. Again we will see in the proofs that in the regimes where $\sC(1)$ does not dominate, a component of double star type will dominate instead, which which exhibits slow mixing.
\end{remark}

\begin{remark}[Transition in the power law] For illustration, we rephrase the main theorems by fixing $\theta$ and varying the tail exponent $\tau  = 1 + 1/\gamma$. 
For $G_N \in \cG_{\beta,\gamma}$ and for the classical dynamics we obtain for  $\theta \in \left(\frac{1}{2}, 1 \right)$, 
\[
\mathbb{E}^{\theta}_{\mu_u}(\tau_{\text{\textnormal{cons}}}|G_N)=
\Theta_{\mathbb{P}}^{\log N}
\begin{cases} N^{\frac{1}{2\tau-4}} & \tau \leq 3 + 2 \left( \frac{1-\theta}{2\theta -1} \right) ,\\
N^{\frac{\theta}{\tau-1}} & \text{otherwise,} 
\\
\end{cases}
\]
and for the discursive dynamics with $\theta \in \left(1,\frac{3}{2}\right)$, this translates to
\[
\mathbf{E}^{\theta}_{\mu_u}(\tau_{\text{\textnormal{cons}}}|G_N)=
\Theta_{\mathbb{P}}^{\log N}
\begin{cases}
N^{\frac{1}{2\tau-4}} & \tau \leq 3+2\left(\frac{\theta-1}{3-2\theta}\right) ,\\
N^{\frac{2-\theta}{\tau-1}} & \text{otherwise} .\\
\end{cases}
\]
From the proof one can see that in both these cases the consensus time on the largest component is dominant only for small $\tau$. If $\theta \in (0,1)$ then for the discursive dynamics we have that 
\[
\mathbf{E}^{\theta}_{\mu_u}(\tau_{\text{\textnormal{cons}}}|G_N)=
\Theta_{\mathbb{P}}^{\log N}
\begin{cases}
N^{\frac{1}{\tau-1}} & \tau \leq 3+ 2\left( \frac{1-\theta}{\theta} \right) ,\\
N^{\frac{2-\theta}{2\tau-4}} & \text{otherwise.} \\
\end{cases}
\]
In this case, one can see from the proofs that the asymptotics for the largest component dominate for  large $\tau$ values.
\end{remark}

\begin{remark}[The supercritical Regime]
On the complete graph the model reduces to the one-dimensional Wright-Fisher model, an example of a well mixed regime. On a supercritical scale free network, too, we expect polylogarithmic mixing times in most cases so that the consensus time on a component $\mathscr{C}$ has order
\begin{equation}\label{order_conjecture}
\Theta^{\log N}_{\mathbb{P}} \left( \frac{1}{\sum_{v \in \mathscr{C}} q(v) \pi(v)^2} \right)
\end{equation}
where $\pi, q$ are the stationary distribution and vertex jump rate of the
random walk obtained by tracing back opinions when restricted to the component $\mathscr{C}$, see Section~\ref{sec:duality} for precise definitions.

The order in Equation \eqref{order_conjecture} is that of the \emph{mean-field} model, see  \cite{sood2008voter} for the mean-field approach. A very general rigorous analysis is made in \cite{cox2016convergence}, under assumption on the mixing and meeting times for the dual chain. However, mixing times in particular are very sensitive to work with and therefore highly model-dependent.

For the Erd\H{o}s-R\'enyi graph we have detailed mixing time and structural results, see  \cite{benjamini2014mixing}. Also on configuration models we have some mixing results, and note that we include a configuration model (with random degrees) in the class of graphs that we are considering in Definition \ref{define_G}. However existing results assume \emph{subpolynomial maximum degree} as in \cite{berestycki2018random}, or a \emph{degree lower bound} as in \cite{abdullah2012cover}.
The conjecture is that these results do extend to general configuration models with power-law degree sequence and so $t_{\text{mix}}=\Theta_{\mathbb{P}}\left( \log^2 N \right)$. Results \cite[(3.21)]{cox2016convergence} and \cite[Lemma 3.17]{aldous-fill-2014}, with discrete-time analogues in \cite{kanade2016coalescence}, would then give the bounds
\[
\Omega_{\mathbb{P}}^{\log N} \left( \frac{1}{\sum_v  q(v) \pi(v)^2} \right)
,\quad
O_{\mathbb{P}}^{\log N} \left( \frac{t_{\text{mix}}}{\sum_v \pi(v)^2} \right),
\]
which in many regimes, those where $\sum_v q(v) \pi(v)^2 \approx \sum_v \pi(v)^2$, are polynomially tight.

In fact, \cite{durrett2010some} conjectures for the $\theta=0$ model via Aldous' ``Poisson Clumping Heuristic'' \cite{aldous2013probability} that the order of the mean consensus time is really the exact polynomial without logarithmic corrections
as found by the mean-field approximation in \cite{sood2008voter}. 
A structural result comparable to \cite{ding2014anatomy} but for the rank one scale-free network would solve the open question of mixing time, but also potentially give a direct handle on meeting time without the logarithmic factors. 
\end{remark}

The remaining paper is organised as follows: 
In Section~\ref{sec:duality}, we describe the (classical) duality of the voter model to coalescing random walks and then develop various tools for  coalescing random walks. This section works whenever the random walks follow the dynamics of a reversible Markov chain, so apply to both our models. 
In Section~\ref{sec:structure}, we derive structural results for subcritical inhomogeneous random graphs. In Section~\ref{voter_models_section}, we combine the structural results with the bounds in Section~\ref{sec:duality} to complete the proofs of Theorems~\ref{class_subcrit} and~\ref{obl_subcrit}.

\section{Duality and bounds on the coalescence time}\label{sec:duality}

In this section, we consider the voter model on an arbitrary finite state space. Moreover, we will discuss the main tool to analyse the voter model, which is the duality to a system of coalescing random walks. In the remaining section, we will then show various bounds on the coalescence time of a system of general random walks.

We will describe a general voter model, where the voters are indexed by $[n]=\{1,\ldots,n\}$ for some $n \in \mathbb{N}$ and the dynamics are governed by a matrix $Q = (Q(i,j))_{i, j \in [n]}$ which is the generator of a continuous-time, reversible Markov chain $(X_t)_{t \geq 0}$ on $[n]$.

Let $O$ be the set of possible opinions, then the $Q$-voter model $(\eta_t)_{t \geq 0}$ with $\eta_t \in O^{[n]}$ 
evolves as follows: for all $i\neq j \in [n]$ at rate $Q(i,j)$ the current state $\eta \in O^{[n]}$ is replaced by
\[ \eta^{i \leftarrow j} (k) = \left\{ \begin{array}{ll} \eta(j) &\mbox{if } k = i \\ \eta(k) &\mbox{if } k \in [n]\setminus \{ i \} .\end{array} \right. \]
In other words, at rate $Q(i,j)$ the voter $i$ copies the opinion from voter $j$.

It is classical that the voter model is dual to a system of coalescing random walks, see~\cite{liggett1985interacting}. 
 The duality can be described via a graphical construction. We start with the graph $\{(j,t) \, : \, j \in [n], t \geq 0\}$ and independent Poisson point processes $(N_{i,j}(t))_{t \in \mathbb{R}}, i \neq j$ (with rates $Q(i,j)$ respectively). If $t_k$ denotes a  jump of $N_{i,j}$ we draw an arrow from $(t_k,j)$ to $(t_k,i)$. Given any initial condition $\eta_0 \in O^{[n]}$, we then let the opinions flow upwards starting at time $0$ and if they encounter an arrow  following the direction of the arrow and replacing the opinion of the voter found at the tip of the arrow.
Now, we fix a time horizon $T > 0$ and  we start with  $n$ random walkers located at each of the points $(j,T), j \in [n]$, then the trajectories follow the graph downwards, following each arrow and if two walkers meet they coalesce. Denote
by $\xi_t^T \subseteq [n] = \{ \xi_t^T(j) \, , \, j \in [n]\}$ the set of positions of these walkers at time $t \geq 0$, where thus $\xi_0^T = [n]$. From this construction, it follows that each walker follows the dynamics of the Markov chain $X$, so we obtain a system of coalescing Markov chains/random walks.
Moreover, one can immediately see that the voter model at time $T$ can be obtained, by tracing the random walk paths backwards, i.e.\  for any $j \in [n]$,
\begin{equation}\label{eq:duality} \eta_T(j) = \eta_0 ( \xi_{T}^T(j)) . \end{equation}

We are interested in general reversible Markov chains, so we do not necessarily assume that the Markov chain is irreducible. However, since $X = (X_t)_{t \geq 0}$ is reversible, we can decompose the state space into its irreducible components, which we 
will denote by $C_1, \ldots, C_k$, so that $X$ restricted to $C_j$ is irreducible.
 In this case, 
for any  $j \in [k]$, we denote the consensus time on the $j$th component by
\[ \tau_{\rm cons}(C_j) = \inf\{ t \geq 0 \, : \, \eta_t |_{C_j} \mbox{ is constant}\}. \]
Then, define the overall consensus time
\[ \tau_{\rm cons} = \max_{j \in [k]} \tau_{\rm cons} (C_j) . \]

Our main interest in this article is in the case when $O = \{0 ,1\}$ and the  initial conditions $\eta_0$ are distributed according to $\mu_u$, the product of Bernoulli$(u)$ measures for some $u \in [0,1]$.
Then, we set
\[ t_{\rm cons}^{(u)} = \E_{\mu_u}(\tau_{\rm cons}). \]

For the duality, it will be easier to consider the voter model where each 
voter starts with a different opinion, i.e.\ $\eta_0 = [n]$. Here,  we define 
\[ t_{\rm cons}^* = \E_{[n]} (\tau_{\rm cons}). \]
For the system of coalescing random walks, we define for each irreducible component $C_j, j \in [k]$,
\[ \tau_{\rm coal} (C_j) = \inf\{ t \geq 0 \, : \, |\xi^T_t| = 1 \} , \]
i.e.\ the first time all walkers in this component have coalesced into a single walker. Moreover, we then define
\[ t_{\rm coal} = \E_{[n]} (\tau_{\rm coal}) , \quad \mbox{where }
\tau_{\rm coal} = \sup_{j \in [k]} \tau_{\rm coal}(C_j) . \]
By duality, we have that if the voter model starts in $\eta_0 = [n]$, then
\[ \p_{[n]} (\tau_{\rm coal} \leq T) = \p_{[n]} ( \tau_{\rm cons} \leq T) , \]
so $\tau_{\rm coal}$ and $\tau_{\rm cons}$ agree in distribution and in particular $t_{\rm coal} = t_{\rm cons}^*$.
As the following lemma shows, we can also get  bounds for $t_{\rm cons}^{(u)}$.

\begin{lemma}[Duality]\label{binary_lower_bound}
In the setting above, we have for any $u \in (0,1)$,
\[ t^{(u)}_{\text{\textnormal{cons}}}
\leq  t_{\text{\textnormal{coal}}}. \]
Suppose additionally that the dual Markov chain is irreducible, then for all $u \in (0,1)$,
\[
t^{(u)}_{\text{\textnormal{cons}}} 
 \geq 2u(1-u) \, t_{\text{\textnormal{coal}}} .
\]
\end{lemma}

\begin{proof} 
By recolouring we can see that we reach consensus from the product Bernoulli  measure $\mu_u$ before we do from unique colours, and hence
\[
t^{(u)}_{\text{\textnormal{cons}}}
\leq
t^*_{\text{\textnormal{cons}}}
=
t_{\text{\textnormal{coal}}}.
\]

For the other direction, suppose that the dual Markov chain is irreducible. Then,  observe from the duality relation \eqref{eq:duality} that
\[
\mathbb{P}_{\mu_u}\left(\eta_T \text{ constant}  \right)
=
\mathbb{P}\left( \mu_u  \text{ constant on } \xi_{T}^T  \right)
=
\mathbb{E} \left(
u^{\left|\xi_{T}^T\right|}
+
(1-u)^{\left|\xi_{T}^T\right|}
\right)
\]
which we can crudely upper bound by considering the event $\{ \left|\xi_{T}^T\right| =1 \}$
\[
\begin{split}
\mathbb{E} \left( u^{\left|\xi_{T}^T\right|}+(1-u)^{\left|\xi_{T}^T\right|} \right)
&\leq
\mathbb{P}\left(
\left|\xi_{T}^T\right| =1
\right)+
\left(u^2+(1-u)^2\right)
\mathbb{P}\left(
\left|\xi_{T}^T\right| \geq 2
\right)\\
&=1-2u(1-u)
\mathbb{P}\left(
\left|\xi_{T}^T\right| \geq 2
\right). 
\end{split}
\]
Therefore we have
\[
\begin{split}
t^{(u)}_{\text{\textnormal{cons}}}
&= \int_0^\infty \p_{\mu_u}( \tau_{\rm cons} \geq t) \, {\rm d} t \\
& = 
\int_0^\infty 1-
\mathbb{E} \left(
u^{\left|\xi_{T}^T\right|}
+
(1-u)^{\left|\xi_{T}^T\right|}
\right)
{\rm d}T\\
&\geq
2u(1-u)
\int_0^\infty
\mathbb{P}\left(
\left|\xi_{T}^T\right| \geq 2
\right)
{\rm d}T
=
2u(1-u)t_{\text{\textnormal{coal}}}, 
\end{split}
\]
where we used irreducibility in the last step. 
\end{proof}

We will control the time $t_{\rm coal}$ until all random walkers have coalesced using the following two bounds in terms of the  two auxiliary quantities that we defined next.
First of all, let $X = (X_t)_{t \geq 0}$ and $Y = (Y_t)_{t \geq 0}$ be two independent reversible Markov chains with generator $Q$. Then, define the (expected) meeting time for $j \in [k]$ as
\[ t_{\rm meet}(C_j)  =  \max_{x,y \in C(j)} \E_{x,y} ( \tau_{\rm meet}) , \quad \mbox{where } \tau_{\rm meet} = \inf\{ t \geq 0 \, : \, X_t = Y_t \} . \]
Moreover, an important role will be played by the (expected) hitting time
defined for $j \in [k]$ as 
\[ t_{\rm hit}(C_j) = \max_{x,y \in C_j} \E_x (T_y) \, , \quad
\mbox{where } T_y = \inf\{ t \geq 0 \, : \, X_t = y \} . \]
Both these quantities give bounds on the coalescence time and thus on the consensus time.

\begin{proposition}\label{prop:coal} With the notation as above, we have that
	\[ \sup_{j \in [k]} t_{\rm meet}(C_j) \leq t_{\rm coal} \leq e( \log n +2) \sup_{j \in [k]} t_{\rm meet}(C_j) .\]
	Moreover, 
	 for any $ j \in [k]$, \[ t_{\rm meet}(C_j) \leq t_{\rm hit}(C_j) .\]
\end{proposition}

\begin{remark} Recall that $t_{\rm coal}$ is defined as
\[ t_{\rm coal} = \E_{[n]}\Big( \sup_{j \in [k]} \tau_{\rm coal} (C_j) \Big) \, \]
so the non-standard part of the statement is that we can take the supremum out of the expectation.
For irreducible chains, the statement is 
 Proposition~14.11 in~\cite{aldous-fill-2014}. However, their proof does not really need this extra assumption. For the convenience of the reader, we repeat the proof below.
 Furthermore, we note that for reducible chains the first bound is shown 
 in~\cite{oliveira2012coalescence} without the $\log n$ factor. 
The stronger bound does not hold without the assumption of irreducibility. 
Indeed, 
by looking at a Markov chain with $n$ components of size $2$ each (e.g.\ with transition rates $1$ within these components), it becomes obvious that the factor $\log n$ in the proposition is sharp.
\end{remark}

\begin{proof}
The reversible Markov chain decomposes into irreducible recurrence classes - write $\mathscr{C}(i)$ for the class containing the state $i$. As in the proof of \cite[Proposition 14.11]{aldous-fill-2014}, consider a walker $W^{(i)}$ independently started in $i$. We have $n(n+1)/2$ meeting times
\begin{equation}\label{eq:meeting_ij}
\tau^{i,j}_{\text{meet}}:=\inf
\left\{
t \geq 0 :
W^{(i)}_t=W^{(j)}_t
\right\}
\end{equation}
for the walkers $1 \leq i \leq j \leq n$, where $\inf \emptyset := \infty$ and $\tau^{i,i} = 0$. Define a function $\operatorname{f}$ which maps all elements in a recurrence class $\mathscr{C}(i)$ to a label $\min \mathscr{C}(i)$ which is of lowest index in that component.
\[
\operatorname{f}: i \mapsto \min \mathscr{C}(i)
\]
Then we can say for the random coalescence time,
\[
\tau_{\text{coal}}:=\max_{i=1}^n \tau_{\text{coal}}(\mathscr{C}(i))\leq \max_{i=1}^n \tau^{i,f(i)}_{\text{meet}}.
\]

We then apply a result for the general exponential tails of hitting times of finite Markov chains \cite[Equation 2.20]{aldous-fill-2014}: from arbitrary initial distribution $\mu$ and for a continuous time reversible chain, for any subset $A\subset V$
\[
\mathbb{P}_{\mu}(T_A>t)\leq \exp \left( - \left\lfloor \frac{t}{e \max_v \mathbb{E}_v T_A } \right\rfloor  \right).
\]

For the meeting time variables this leads to
\[
\mathbb{P}(\tau^{i,j}_{\text{meet}}>t)\leq \exp\left( - \left\lfloor \frac{t}{e t_{\text{meet}}} \right\rfloor \right) . 
\]
We can deduce that
\[\mathbb{P}(\tau_{\text{coal}}>t)
\leq \sum_{i=1}^n \mathbb{P}\left(\tau^{f(i),i}_{\text{meet}}>t\right)
\leq n\exp\left( - \left\lfloor \frac{t}{e t_{\text{meet}}} \right\rfloor \right).
\]
Finally, we conclude  as in \cite[Proposition 14.11]{aldous-fill-2014} by integrating  to get
\[
t_{\rm coal}\leq
\int_0^\infty
1 \wedge
\left(
n e \exp\left( - \frac{t}{e t_{\text{meet}}} \right)
\right) {\rm d} t
=
e \left(
2+\log n
\right) t_{\text{meet}}, 
\]
which proves the first claim.

The second claim  of the proposition is \cite[Proposition 14.5]{aldous-fill-2014}. 
\end{proof}

In particular, Proposition~\ref{prop:coal} allows us to to bound the consensus time by bounding either hitting times or meeting times for an irreducible chain. We start by collecting the bounds on hitting times in Section~\ref{ssec:hitting} and continue with the bounds on meeting times in Section~\ref{ssec:meeting}

\subsection{Bounds on hitting times}\label{ssec:hitting}

Throughout the following two sections $X = (X_t)_{t \geq 0}$ will be a reversible, irreducible Markov chain with state space $[n] = \{1, \ldots, n\}$ and transition rates given by a matrix~$Q$. Moreover, we denote by  $\pi = (\pi(i))_{i \in [n]}$ the invariant measure of $X$.

In this case, as there is only one irreducible component the (expected) hitting time is  defined as 
\[ t_{\rm hit} = \max_{k, j \in [n]} \E_k (T_j) . \]
For our bound on the hitting time, we will make use of the well-known correspondence between 
Markov chains and electric networks, see e.g.~\cite{aldous-fill-2014, wilmer2009markov}.
In this context, we associate to $Q$ a graph $G_Q$ with vertex set $[n]$ and connect $i$ and $j$ by an edge, written $i \sim j$, if the conductance $c(ij)$ is nonzero, where the conductance is defined as
\begin{equation}\label{conductance_definition}
c(ij) := \pi(i) Q(i,j) = \pi(j) Q(j,i).
\end{equation}

This is also known as the \emph{ergodic flow} of the edge. Moreover, the interpretation as an electric network lets us define the effective resistance between two vertices $i,j \in [n]$, denoted $\mathcal{R}(i \leftrightarrow j)$, as in \cite[Chapter 9]{wilmer2009markov}.

To state the following proposition, we also define  ${\rm diam}(Q)$ to be the diameter in the graph theoretic 
sense for the graph obtained from $Q$ as above.
The proof uses the representation of the effective resistances in terms of the Markov chains, combined with Thomson's principle.

\begin{proposition}[Conductance bounds]\label{prop:max_resistance}
Let  $(X_t)_{t \geq 0}$ be a reversible, irreducible Markov chain on $[n]$ with associated conductances $c$.
Let $P_{i,j}$ be a path from $i$ to $j$ in $G_Q$ and denote by $E(P_{i,j})$ the set of edges in $P_{i,j}$. Then
\[
\mathbb{E}_i \left( T_j \right) +\mathbb{E}_j \left( T_i \right) \leq \sum_{e \in E(P_{i,j})}\frac{1}{c(e)}.
\]
In particular, we have
\[ t_{\rm hit} \leq {\rm diam} (Q) \max_{i\sim j \in [n]} \frac{1}{c(ij)} . \]
\end{proposition}

\begin{proof}
Let $T_i^+ = \inf\{ t > 0 \, : \, X_t = i, \lim_{s \uparrow t} X_s \neq i \}$ be the return time to state $i$.
From \cite[Proposition 9.5]{wilmer2009markov}, 
\[
\mathcal{R}(i \leftrightarrow j)=\frac{1}{c(i) \mathbb{P}_i(T_j<T_i^+)}, 
\]
where $c(i) = \sum_{j \sim i}c({ij})$ is the conductance around a vertex. We further have from \cite[Corollary 2.8 (continuous time version)]{aldous-fill-2014}
\[
\mathbb{E}_i \left( T_j \right) +\mathbb{E}_j \left( T_i \right) =\frac{1}{\pi(i)q(i) \mathbb{P}_i(T_j<T_i^+)},
\]
where $q(i) = - Q(i,i)$ is the walker speed at $i$, and by the choice of $c$ these expressions are equal. Finally by Thompson's Theorem (which describes monotonicity of effective resistances with respect to edge resistances)
\[
\begin{split}
\mathbb{E}_i \left( T_j \right) +\mathbb{E}_j \left( T_i \right) ={\mathcal{R}(i \leftrightarrow j)} &\leq {\mathcal{R}(i \leftrightarrow j \text{ through } P_{i,j})}\\
&= \sum_{\{u,v\} \in E(P_{i,j})}\frac{1}{c(uv)},
\end{split}
\]
which gives the required bound.
\end{proof}

\subsection{Bounds on meeting times}\label{ssec:meeting}

In this section, we continue to use the notation from the beginning of Section~\ref{ssec:hitting}. In particular, $(X_t)_{t \geq 0}$ is  a reversible, irreducible Markov chain on $[n]$ with transition rates given by $Q$ and 
invariant measure $\pi$.

It will often be easier to work with the (expected) meeting time when both chains are started in 
the invariant measure, i.e.\ we define
\[ t_{\rm meet}^\pi :=  \sum_{i, j \in [n]} \pi(i) \pi(j) \E_{i,j}(\tau_{\rm meet}) .\]
In order to make the connection to $t_{\rm meet}$, we will need the time it takes to reach stationarity. There are competing definitions of the distance from stationarity, both of which are required to apply the literature results.

\begin{definition}\label{stat_dist}
For a Markov chain on $[n]$
\[
d(t):=\frac{1}{2}\max_{x \in [n]}\| p^{(t)}_{x,\cdot}-\pi(\cdot) \|_1,
\]
\[
\bar{d}(t):=\frac{1}{2}\max_{x,y \in [n]}\| p^{(t)}_{x,\cdot}-p^{(t)}_{y,\cdot} \|_1.
\]
\end{definition}

The mixing time $t_{\rm mix}$ is then defined as 
\[
t_{\text{\textnormal{mix}}}:=\min \Big\{ t\geq 0 : d(t)\leq \frac{1}{4}\Big\},
\]
and the mixing time from a point $i \in [n]$ as
\[
t_{\text{\textnormal{mix}}}(i):=\min \Big\{ t\geq 0 : \big\|p_{i\cdot}^{(t)}-\pi \big\|_1 \leq \frac{1}{2}\Big\}.
\]
Closely related to the mixing time is the relaxation time
\[ 
t_{\text{\textnormal{rel}}}:=\max \left\{ \frac{1}{\lambda} : \lambda \text{ a positive eigenvalue of } -Q \right\} \]
and we describe the relationship in the following lemma, as standard references are either in discrete time or using different definitions. 

\begin{lemma}\label{mixing_and_relaxation}
\[
t_{\text{\textnormal{mix}}} \geq \frac{t_{\text{\textnormal{rel}}}}{1+\frac{1}{\log 2}}.
\]
\end{lemma}

\begin{proof}

By the previous definitions of the mixing time and stationarity distances, we have that if 
$d(t_{\text{\textnormal{mix}}})  = \frac{1}{4}$, then it follows that  $\bar{d}(t_{\text{\textnormal{mix}}})\leq \frac{1}{2}$. 
Therefore, by the submultiplicativity of $\bar{d}$ shown in \cite{aldous-fill-2014} for any $C\geq 1$ we have 
\[
\bar{d}(C t_{\text{\textnormal{mix}}})\leq 2^{-\lfloor C\rfloor}\leq 2^{1-C}.
\] 
The right hand side is less than $e^{-1}$ when $
C=1+\frac{1}{\log 2}$. Therefore, by  \cite[Lemma 4.23]{aldous-fill-2014} we can deduce that
$
C t_{\text{\textnormal{mix}}} \geq t_{\text{\textnormal{rel}}}
$
and the claim follows.
\end{proof}

\begin{proposition}
\begin{enumerate}[label={(\alph*)},ref={\theproposition~(\alph*)}]
	\item \label{lower_meeting_bound}
\[
t_{\text{\textnormal{meet}}}^{\pi} \geq \frac{(1-\sum_{i \in [n]} \pi(i)^2)^2}{4\sum_{i \in [n]} q(i) \pi(i)^2}, 
\]
where $q(i) = - Q(i,i)$.
\item \label{conductance_theorem} 
There exists an absolute constant $c_{\rm cond} > 0$ such that
\[
 t_{\rm{meet}} \geq 
c_{\rm cond} \left(
\max_{A \subset [N]}
\frac{\pi(A) \pi(A^c)}{\sum_{x \in A} \sum_{y \in A^c}c(xy)}
\right).
\]
\end{enumerate}
\end{proposition}

\begin{proof}
For Part (a) see Remark 3.5  in~\cite{cox2016convergence}.

(b) From the standard coupling bound for mixing times seen in \cite[Theorem 9.2]{aldous-fill-2014} and with $\tau_{\rm meet}^{ij}$ as in~\eqref{eq:meeting_ij},
\[
d(t)
\leq \max_{i,j} \mathbb{P}\left( \tau_{\text{meet}}^{i,j}>t \right)
\leq \exp\left( - \left\lfloor \frac{t}{e t_{\text{meet}}} \right\rfloor \right)
\leq \exp\left(1 - \frac{t}{e t_{\text{meet}}} \right)
\]
where the second inequality is from \cite[Equation (2.20)]{aldous-fill-2014}.
So by integrating
\begin{equation}\label{mixing_and_meeting}
\frac{1}{4} t_{\text{mix}}\leq \int_0^{\infty} d(t) {\rm d} t \leq e^2 t_{\text{meet}} . 
\end{equation}

Because $c(xy)=\pi(x)Q(x,y)$, by \cite[Corollary 4.37]{aldous-fill-2014},
\[
\max_{A \subset [N]}
\frac{\pi(A) \pi(A^c)}{\sum_{x \in A} \sum_{y \in A^c}c(xy)} \leq t_{\text{\textnormal{rel}}}.
\]

Combining this with Equation \eqref{mixing_and_meeting} and Lemma \ref{mixing_and_relaxation} proves the claim.
\end{proof}

\cite[Corollary 1.2]{peres2017intersection} has the consequence that, for some universal $C>0$, $t_{\rm mix} \leq C \min_i t_{\rm hit}(i)$. We present a simple proof of this fact for the convenience of the reader and to give an explicit constant factor.

\begin{lemma}\label{mixing_below_central_hitting} For any $i \in [n]$,
\[
t_{\text{\textnormal{mix}}}(i)\leq 2\mathbb{E}_{\pi}\left( T_i \right).
\]
\end{lemma}

\begin{proof}
Let $i \in [n]$, then by Cauchy-Schwarz we have that
\[ \left\|p_{i\cdot}^{(t)}-\pi\right\|_1^2 = \sum_{j \in [n]} \Big| \frac{p_{i j}^{(t)}}{\pi(j)} - 1\Big| \pi(j) 
\leq \bigg( \sum_{i \in [n]} \Big| \frac{p_{i j}^{(t)}}{\pi(j)} - 1\Big|^2 \pi(j) \bigg)^2 = 
\Big\|\frac{p_{i\cdot}^{(t)}}{\pi}-1\Big\|_\pi^2.
\]

To simplify the right hand side, we use reversibility to obtain
\[
\Big\|\frac{p_{i\cdot}^{(t)}}{\pi}-1\Big\|^2_\pi
=-1+\sum_j \frac{\left( p_{ij}^{(t)} \right)^2}{\pi(j)}
=-1+\frac{1}{\pi(i)}\sum_j p_{ij}^{(t)}p_{ji}^{(t)}
=-1+\frac{p_{ii}^{(2t)}}{\pi(i)}.
\]

Now, by \cite[Lemma 2.11]{aldous-fill-2014}, we have that for any $t \geq 0$,
\[
\begin{split}
\mathbb{E}_\pi \left( T_i \right) &= \int_0^\infty \left( -1+\frac{p_{ii}^{(s)}}{\pi(i)} \right) {\rm d}s 
\geq 2 t \left( -1+\frac{p_{ii}^{(2t)}}{\pi(i)} \right), 
\end{split}
\]
because the integrand is non-increasing \cite[Equation 3.40]{aldous-fill-2014}.
Combining these inequalities, we have for $t > 0$,
\begin{equation}\label{l1_bound}
\Big\|p_{i\cdot}^{(t)}-\pi\Big\|_1 \leq
\Big\|\frac{p_{i\cdot}^{(t)}}{\pi}-1\Big\|_\pi \leq
\sqrt{\frac{\mathbb{E}_\pi \left( T_i \right)}{2 t}} . 
\end{equation}
Hence, if $t$ is such that
\[
\frac{\mathbb{E}_\pi \left( T_i \right)}{2 t}\leq \frac{1}{4}, 
\]
then we can deduce that $t_{\rm mix}(i) \leq t$, which completes the proof.
\end{proof}

\begin{proposition}\label{corr:strong_stationary_time}
For two independent copies $(X_t)_t$ and  $(Y_t)_t$ of a Markov Chain on $[n]$, and any state $s \in [n]$, we find
\[
t_{\rm mix} \leq 16 \, t_{\text{\textnormal{hit}}}(s)
\]
and further we can construct a time for the product chain with
\[
\mathbb{E}_{x,y}(\cS)\leq 188  \,  t_{\rm hit}(s)
\]
which is a strong stationary time in the sense that for any $t \geq 0$ we have
$
\mathcal{L}(X_{t+\mathcal{S}},Y_{t+\mathcal{S}})=\pi \otimes \pi
$
and, further, $(X_{t+S},Y_{t+S})_{t \geq 0}$ and $\mathcal{S}$ are independent.
\end{proposition}

\begin{proof}
Define the time $M_1 = 8 t_{\rm hit}(s)$. Then, by Markov's inequality
\[
\max_x \mathbb{P}_x (T_s \geq M_1)
\leq
\max_x \frac{\mathbb{E}_x (T_s)}{M_1}
=
\frac{t_{\text{\textnormal{hit}}}(s)}{M_1}
=
\frac{1}{8}
\]
so we will hit $s$ in the timeframe $[0,M_1]$ with probability at least $\frac 78$. Define also the time $M_2$ by
\[
\frac{\mathbb{E}_\pi \left( T_s \right)}{2 M_2} = \frac{1}{16}
\]
then by recalling equation \eqref{l1_bound} we have that
\[ \frac{1}{2} \| p_{s,\cdot}^{(M_2)}  - \pi(\cdot) \|_1 \leq \frac{1}{8}. \]
Hence, by distinguishing the cases of hitting $s$ by $M_1$, or not, we obtain that $d(M_1 + M_2) \leq \frac{1}{4}$. 
Thus, 
\[
t_{\text{\textnormal{mix}}}
\leq
M_1+M_2
=
8 \, t_{\text{\textnormal{hit}}}(s)+8  \, \mathbb{E}_\pi \left( T_s \right)
\leq
16 \,  t_{\text{\textnormal{hit}}}(s).
\]

It remains to prove the second claim. By Theorem 1.1 in \cite{fill1991time} we can construct a strong stationary time with
\[
\mathbb{P}_s\left(\cS_X>t\right)=\operatorname{sep}_s(t)=:1-\min_{j \in [n]} \frac{p^{(t)}_{s, j}}{\pi(j)}.
\]
Then we recover several definitions and results from \cite{aldous-fill-2014} given in Definition \ref{stat_dist}.
These various definitions of distance from stationarity satisfy
\[
\operatorname{sep}_s(2t)\leq\max_{v \in [N]}\operatorname{sep}_v(2t)<2 \, \bar{d}(t)\leq 4 \,  d(t) . 
\]
Therefore, we have that
\[
\tau_1:=\min \left\{ t: \bar{d}(t) \leq \frac{1}{2}\right\}\leq \min \left\{ t: d(t) \leq \frac{1}{4}\right\} = t_{\text{mix}} .
\]
Then we use that $\bar{d}$ is submultiplicative to obtain
\[
\bar{d}(t)\leq 2^{-\lfloor t/\tau_1 \rfloor }\leq 2^{-\lfloor t/t_{\text{mix}} \rfloor }.
\]
Thus, we can bound the expectation of the time to stationarity
\[
\mathbb{E}_s \left(\cS_X\right)
=2\int_0^\infty \operatorname{sep}_s(2t) {\rm d}t
\leq 4 \int_0^\infty 2^{1- t/t_{\text{mix}}} {\rm d}t
=\frac{8 \,  t_{\text{mix}}}{\log 2}
\leq \frac{64 \,  t_{\text{\textnormal{hit}}}(s)}{\log 2}.
\]

This becomes a strong stationary time for $(X_t)_t$ with $X_0=x$ by constructing another time $\tilde{\cS}_X$ which simply waits for the event 
when the walker hits $s$, and then waits for $\cS_X$. Thus
\[
\mathbb{E}_x\left(\tilde{\cS}_X \right)\leq t_{\rm hit}(s)
+
\frac{64  \, t_{\text{\textnormal{hit}}}(s)}{\log 2}
<
94 \,  t_{\rm hit}(s).
\]

We construct the symmetric time $\tilde{\cS}_Y$ for $(Y_t)_t$ and then finally our object is the time
\[
\cS:= \tilde{\cS}_X \vee \tilde{\cS}_Y
\]
so that
\[
\mathbb{E}_{x,y}(\cS) \leq
\mathbb{E}_x\left(\tilde{\cS}_X \right)+\mathbb{E}_y\left(\tilde{\cS}_Y \right)
\leq 188 \,  t_{\rm hit}(s),
\]
as claimed.
\end{proof}

\begin{proposition}\label{tree_meeting_theorem}
For any state $s \in [n]$
\begin{equation*}
t_{\text{\textnormal{meet}}} \leq \frac{189 \,  t_{\text{\textnormal{hit}}}(s)}{\pi(s)}.
\end{equation*}
\end{proposition}

\begin{proof}
From any configuration of two walkers, we can apply Proposition \ref{corr:strong_stationary_time} to construct a strong stationary time $\cS$ with $\mathbb{E}(\cS)\leq 188  \,  t_{\rm hit}(s) $. Then, wait for $(X_t)_t$ to hit $s$, which in expectation takes an additional time period of length $t_{\rm hit}(s)$.

On $(X_t)_t$ hitting $s$, $(Y_t)_t$ is still in independent stationarity, and so we have exactly probability $\pi(s)$ to meet at that instant. Otherwise, we restart the argument with mixing and hitting periods to get another chance to meet at $s$.

Thus we have to repeat the attempt no more than $\operatorname{Geom} (1/\pi(s))$ times, and each attempt conditionally expects to take no longer than $188  \,  t_{\rm hit}(s)+ t_{\rm hit}(s)$.
\end{proof}

\begin{remark}
We find the following illustrative discrete-time bound in \cite{kanade2016coalescence}

\begin{equation*}
t_{\text{\textnormal{meet}}}^{\pi} = O\left(\frac{t_{\text{\textnormal{mix}}}}{||\pi||_2^2}\right)
\end{equation*}

which, while appearing a better bound, is commonly not so for Markov chains on trees. The mixing time for a Markov chain on a tree (which must always be a reversible chain) is always the hitting time of a \emph{central} vertex, i.e. one with
\[
\mathbb{E}_{\pi}T_c={\operatorname{min}}_{v \in [n]}\mathbb{E}_{\pi}T_v. 
\]
Then,
\[
t_{\text{\textnormal{mix}}} = \Theta \left( t_{\text{\textnormal{hit}}}(c) \right)
\]
and so because $\|\pi\|_2^2\leq\|\pi\|_\infty$, Proposition \ref{tree_meeting_theorem} will often give a tighter bound.
\end{remark}

In the following, we will need
the following large deviations result  given in \cite{saloff1997lectures}.

\begin{theorem}\label{thm:markov_chain_large_deviations}
For any finite, irreducible continuous-time Markov chain $(X_t)_t$ with initial stationary distribution $\pi$, and any function on the state space $f$ with
\[
\langle f, \pi \rangle =0, \qquad ||f||_\infty \leq 1,
\]
we have for $x>0$ 
\[
\mathbb{P}_\mu \left(
\frac{1}{t} \int_0^t f(X_s) {\rm d}s > x
\right)
\leq
||\mu/\pi||_2 \exp \left(
-
\frac{x^2 t}{10 \,  t_{\text{\textnormal{rel}}}}
\right),
\]
where $\mu$ is an arbitrary distribution on the state space.
\end{theorem}

We now use the concept of the chain $(X_t)_{t \geq 0}$ \emph{observed on a subset} $V \subset N$ described in Section 2.7.1 in~\cite{aldous-fill-2014}: define a  clock process
\[
U(t):=\int_0^t \mathbbm{1}_{V} \left( X_s \right) {\rm d}s, 
\]
with generalised right-continuous inverse $U^{-1}$. Then the partially observed chain $(P_t)_{t \geq 0}$ is defined for any  $t \geq 0$ via
\[
P_t:=X_{U^{-1}(t)}.
\]

This corresponds to the deletion of states in $V^c$ from the history of $(X_t)_{t\geq 0}$, and so it can be shown that  $(P_t)_{t \geq 0}$ is Markovian and has the natural stationary distribution
\[
\frac{\pi(\cdot) \mathbbm{1}_{V} (\cdot)}{\pi(V)}.
\]

Then, we can define the random \emph{subset meeting time} $\tau^\pi_{\text{\textnormal{meet}}}(A)$ analogously to $\tau^\pi_{\text{\textnormal{meet}}}$ except for the partially observed product chain on $A \times A$ rather than the full chain. Similarly,  $t^\pi_{\text{\textnormal{meet}}}(A)=\mathbb{E}(\tau^\pi_{\text{\textnormal{meet}}}(A))$.

\begin{theorem}\label{partial_meeting} For any $A \subset [n]$,
\[
t_{\rm meet}
\leq
188  t_{\rm hit}(s)
+
\frac{2 \, t^\pi_{\text{\textnormal{meet}}}(A)}{\pi(A)^2}
+
\frac{1568 \,  t_{\rm hit}(s)}{\pi(A)^4}
.
\]
\end{theorem}

\begin{proof}
We first prove the claim
\begin{equation}\label{eq:claim1}
t^\pi_{\text{\textnormal{meet}}}
\leq
\frac{2 \, t^\pi_{\text{\textnormal{meet}}}(A)}{\pi(A)^2}
+
\frac{98 \,  t_{\text{\textnormal{mix}}}}{\pi(A)^4}.
\end{equation}

Consider two independent copies $(X_t, Y_t)_t$ of the stationary chain, such that in particular, for any 
$ t \geq 0$ we have that $\mathcal{L}(X_t, Y_t)= \pi \otimes \pi$.
Define the time-change
\[
U(t):=\int_0^t \mathbbm{1}_{A \times A} \left( X_s,Y_s \right) {\rm d}s .
\]
Then, the product chain $(\tilde X_t, \tilde Y_t)_{t \geq 0}$ observed on $A \times A$ satisfies $(\tilde X_t, \tilde Y_t) = (X_{U^{-1}(t)}, Y_{U^{-1}(t)})$ for any $t \geq 0$. 
Therefore, we have that for any $t \geq 0$,
\[ \p ( U(\tau_{\rm meet}^\pi) \geq t ) \leq \p( \tau_{\rm meet}^\pi(A) \geq t) ,\]
since a meeting might also happen outside $A$.
In particular, we can deduce that 
\begin{align}\label{eq:2605-1} 
\E(\tau_{\rm meet}^\pi)   & = 
\int_0^\infty \p( \tau_{\rm meet}^\pi >t) \,{\rm d} t \notag\\
& \leq \int_0^\infty \p \Big( U( \tau_{\rm meet}^\pi) > U(t); U(t) \geq \frac{\pi(A)^2}{2} t\Big){\rm d }t  + \int_0^\infty\p  \Big(U(t) \leq \frac{\pi(A)^2}{2} t \Big) \,{\rm d} t\notag \\
&  \leq\int_0^\infty \p \Big( \tau^\pi_{\rm meet}(A) \geq  \frac{\pi(A)^2}{2} t\Big){\rm d }t  + \int_0^\infty\p  \Big(U(t) \leq \frac{\pi(A)^2}{2} t \Big) \,{\rm d} t \notag\\
& \leq \frac{2}{\pi(A)^2} t_{\rm meet}^\pi(A) + \int_0^\infty\p  \Big(U(t) \leq \frac{\pi(A)^2}{2} t \Big) \,{\rm d} t .
\end{align}
It remains to estimate the second integral on the right hand side.

 For this purpose, we apply Theorem \ref{thm:markov_chain_large_deviations}
 to the function
$ f :=
\mathbbm{1}_{ A \times A }
-\pi(A)^2$
to obtain
\[
\mathbb{P}_\pi \left(
\frac{1}{t}\int_0^t \mathbbm{1}_{A \otimes A} \left( X_s, Y_s \right) {\rm d}s
-\pi(A)^2
<-x
\right)
\leq
\exp \left(
-
\frac{x^2 t}{10  \, t_{\text{\textnormal{rel}}}}
\right)
\]
and hence
\[
\begin{split}
\mathbb{P}\left(U(t) \leq \frac{\pi(A)^2}{2}t\right) 
&= \mathbb{P}\left( \frac{U(t)}{t}-\pi(A)^2 \leq -\frac{\pi(A)^2}{2} \right) 
\leq \exp \left( - \frac{t \pi(A)^4}{40 \,  t_{\text{rel}}}\right) . 
\end{split}
\]
We deduce that
\[ \int_0^t  \p \left( U(t) \leq \frac{\pi(A)^2}{2}t \right)
\leq \frac{40 \,  t_{\rm rel}}{\pi(A)^4}. \]
Moreover, 
Lemma \ref{mixing_and_relaxation} gives 
\[
40 \,  t_{\rm rel}
\leq
40 \left(1+\frac{1}{\log 2}\right) t_{\rm mix}
< 98 \,  t_{\rm mix} . 
\]
Combining these estimates with~\eqref{eq:2605-1} gives the claim~\eqref{eq:claim1}.

To obtain the  statement of the theorem, recall from Proposition \ref{corr:strong_stationary_time} that there exists a strong stationary time $\cS$ such that 
\[
t_{\rm mix} \leq 16  \, t_{\text{\textnormal{hit}}}(s)
\quad \text{and} \quad
\mathbb{E}_{x,y}(\cS)\leq 188  \,  t_{\rm hit}(s).
\]
Using the stationary time constructed in this corollary gives the  bound
\[
t_{\rm meet}\leq
\max_{x,y \in [n]}
\mathbb{E}_{x,y}(\cS)
+
t^\pi_{\rm meet},
\]
which together with~\eqref{eq:claim1} proves the theorem.
\end{proof}

\section{Structural results for subcritical random graphs}\label{sec:structure}

In this section, we collect some of the structural results on subcritical inhomogeneous 
random graphs that we will need later on in Section~\ref{voter_models_section}. 
Some of these results are known, but as the literature on subcritical inhomogeneous random 
graphs is less developed than for supercritical random graphs, we have to prove
the more specialised ones.

Let $G_N \in \mathcal{G}_{\beta,\gamma}$.
Denote by $\comps$  the set of (connected) components of $G_N$. 
For any $\sC \in \comps$ we write the graph as $(V(\sC),E(\sC))$ and denote by $|\sC|:=|V(\sC)|$ the number of vertices in $\sC$. Moreover, we let $\sC(i)$ denote the component containing vertex $i$.
Throughout this section, we will use the notation 
\[ K_\gamma:= N^\frac{1-2\gamma}{2-2\gamma} \log N , \]
and  call a component $\sC \in \comps$ \emph{big} if 
$\sC = \sC(i)$ for some $i \leq K_\gamma$. Otherwise, the component is called \emph{small}. 
Moreover, we define the collection of all vertices lying in big components as
\[ V_{\rm big} := \bigcup_{i \leq K_\gamma} V(\sC(i)) . \]

The first proposition is a standard result on the (componentwise) diameter. 

\begin{proposition}\label{prop:diameter}
For $G_N \in \cG_{\beta,\gamma}$ with $\beta + 2 \gamma < 1$, we have that 
	\[ \diam (G_N) := \sup_{\sC \in \comps}  \diam (\sC) = O_{\p} ( \log N ). \]
\end{proposition}

As we will see later on, for the classical voter model, the invariant measure of the associated random walk is normalized by $\sum_{z \in \sC(k)} \de(z)^{\theta-1}$, so  that in the following we collect various bounds on
$\sum_{z \in \sC(k)} \de(z)$.

\begin{proposition} 
For $G_N \in \cG$ with $\beta+2\gamma<1$, with high probability,
\begin{enumerate}[label={(\alph*)},ref={\theproposition~(\alph*)}]
\item \label{prop:big_sum_of_degrees}
\[
\max_{k \leq K_\gamma}
\frac{\sum_{z \in \sC(k)} \de(z)}{(N/k)^\gamma} \leq \log N.
\]
\item \label{prop:small_sum_of_degrees}
\[
\max_{i\notin V_{\rm big}} 
\sum_{v \in \mathscr{C}(i)} \operatorname{d}(v)
=O_{\mathbb{P}}^{\log N}
\left(N^{\frac{\gamma}{2-2\gamma}}\right).
\]
\end{enumerate}
\end{proposition}

For the largest component this result is not optimal as we lose a $\log$ factor, 
see also~\cite[Theorem 1.1]{janson2008largest}, but the latter result does not cover the other components.

As a next result, we need that the large degrees $\de(i)$ are well approximated by their means.
Also, we need to know that for each of the vertices with large degree, a positive proportion of its neighbours has degree $1$. 
One of the challenges in the proof is that we need these bounds uniformly over all big components.

\begin{proposition}\label{prop:stars_and_leaves}
For $G_N \in \cG_{\beta, \gamma}$ with $\beta+2\gamma<1$ 
the following statements hold: 
\begin{enumerate}[label={(\alph*)},ref={\theproposition~(\alph*)}]
\item \label{prop:star_degrees}
\[
\min_{k \leq K_\gamma}
\frac{\operatorname{d}(k)}{(N/k)^\gamma}=\Omega_{\mathbb{P}}(1),
\quad \quad
\max_{k \leq K_\gamma}
\frac{\operatorname{d}(k)}{(N/k)^\gamma}=O_{\mathbb{P}}(1).
\] 
\item \label{prop:leaf_counts} For any $k \in [N]$, let $L_k$ be the number of neighbours of $k$ of degree $1$, then we have
\[
\min_{k \leq K_\gamma}
\frac{|L_k|}{\operatorname{d}(k)}=\Omega_{\mathbb{P}}(1) .
\]
\end{enumerate}
\end{proposition}

\begin{definition}\label{branch_defn} For $G_N \in \cG_{\beta,\gamma}$ and any component  $\sC \in \comps$, 
we define the set of \emph{branches} $\cB(\sC)$ of $\sC$ as the set of connected components of the subgraph of $\sC$ induced by the vertex set $V(\sC) \setminus \{ i\}$, where $i = \min V(\sC)$.
\end{definition}

We will use this definition specifically in the context when $\sC$ is a tree, so 
that this terminology makes sense.
The next lemma states that big components are trees and have branches that are small (at least when compared to the largest components of order $N^\gamma$).

\begin{lemma}\label{le:subcritical_branch_control}
For $G_N \in \cG_{\beta, \gamma}$ with $\beta+2\gamma<1$, with high probability 
every \emph{big} component is a tree. On this event we have that
\[
\max_{k \leq K_\gamma }
\max_{B \in \cB(\sC(k))}
\sum_{v \in B} \de(v)
=
\max_{k \leq K_\gamma}
\max_{B \in \cB(\sC(k))}
(2|B|-1)
=
O_{\mathbb{P}}^{\log N}\left( N^{\frac{\gamma}{2-2\gamma}} \right).
\]
\end{lemma}

The following claim for the empirical moment of the degree distribution of $\sC(1)$ will allow us to demonstrate a lower bound on this component for certain parameters of both the classical and discursive models.

\begin{lemma}\label{le:empirical_moment}
For $G_N \in \cG_{\beta,\gamma}$ with $\beta+2\gamma<1$, then for any $\eta \geq 1$
\[
\sum_{v \in \mathscr{C}(1)}\operatorname{d}(v)^\eta =\Theta^{\log N}_{\mathbb{P}}\left( N^{\gamma\eta} \right).
\]
\end{lemma}

Two of our lower bounds require the existence of a  `double star' component together with a suitable bound on the empirical moment.

\begin{proposition}\label{prop:existence_simple_double_star}
For $G_N \in \cG_{\beta,\gamma}$ with $\beta+2\gamma<1$ there exists with high probability a tree component containing two adjacent vertices $x,y \in K_\gamma$ such that
\[
\de(x) \mbox{ and }  \de(y) \mbox{ are }\Theta_{\mathbb{P}}^{\log N} \left(
N^\frac{\gamma}{2-2\gamma}
\right)
\]
and further for any $\eta\geq 1$ 
\[
\sum_{v \in \sC(x)} \de(v)^\eta = \Theta_{\mathbb{P}}\left(
N^\frac{\gamma \eta}{2-2\gamma}
\right).
\]
\end{proposition}

The final proposition of this section states that we can always find a ``long double star'' in $G_N$, i.e. two vertices with degree of order at least $N^{\gamma/(2-2 \gamma)}$ that are connected via a short path with two intermediate vertices of degree $2$ each. The path having length at least $3$ is important for the discursive voter model dynamic. 

\begin{proposition}\label{prop:existence_long_double_star}
With high probability any $G_N \in \cG_{\beta,\gamma}$ with $\beta + 2\gamma <1$ contains a path
$\mathcal{P}=(v_1,v_2,v_3,v_4)$
such that:
\begin{enumerate}[label={(\alph*)},ref={\theproposition~(\alph*)}]
\item $\de(v_2)=\de(v_3)=2$. 
\item $\{ v_1, v_4 \} \subset [K_\gamma]$ (and hence the component is a tree)
\item $\de(v_1), \de(v_4) =\Theta_{\mathbb{P}}^{\log N} \left(
N^\frac{\gamma}{2-2\gamma}
\right).$
\end{enumerate}
\end{proposition}

In the remaining part of this section, we will prove these results. An essential tool will be a coupling 
with a branching process that we set up in Section~\ref{Branching_section}. Then in Section~\ref{subcrit_cpts_section} we will prove the structural results stated above.

\subsection{Coupling with a branching process}\label{Branching_section}

By Remark~\ref{rem:IRG}(c), we have some flexibility for which model in the class $\mathcal{G}_{\beta,\gamma}$ to show our results. For most of our proofs, we will prove the statements for the simple Norros-Reittu (SNR) model, i.e.\ where
edges are present independently with probabilities
\begin{equation}\label{nr_edge_probabilities}
q_{i,j}=1-e^{-p_{i,j}}. 
\end{equation}
The reason for this choice is the close relation with the standard multigraph Norros-Reittu  (MNR) model. 

In the MNR multigraph $G_N^{\rm NR}$ each vertex $i \in [N]$ has weight $w(i)>0$ and independently for each pair $\{i,j\}$ with $i,j \in [N]$, the number of edges between $i$ and $j$ has the distribution
\[
\operatorname{
Pois}\left(\frac{w(i) w(j)}{w([N])}\right),
\]
where $w([N])=\sum_{i=1}^N w(i)$ is the total weight and where we write ${\rm Pois}(\mu)$ for a Poisson distribution with parameter $\mu > 0$. Note this graph model not only has multiple edges, but also allows for self-loops.

The SNR model with edge probability as in~\eqref{nr_edge_probabilities} is then obtained by first choosing
\begin{equation}\label{weight_defn}
w(i):=\sum_{j=1}^N \beta N^{2\gamma-1} i^{-\gamma} j^{-\gamma}
\sim \frac{\beta}{1-\gamma}\left( \frac{N}{i} \right)^\gamma ,
\end{equation}
and then collapsing all multi-edges to simple edges and deleting the loops.

The MNR model is particularly nice, because it allows for an exact coupling with a 
two-stage Galton-Watson process with thinning and cycle creation.
Our construction here extends the coupling introduced in~\cite{norros2006conditionally} (see also  \cite{van2016random}) by also keeping track of the number of edges, so that we can also control when we create cycles. 

Define  the \emph{mark distribution} to be the random variable $M$ on $[N]$ which chooses a vertex biased proportional to its weight
\[
\mathbb{P}(M=m)\propto w(m) \mathbbm{1}_{m\in [N]}
\propto m^{-\gamma} \mathbbm{1}_{m\in [N]}
\]
so that if $W_N$ is the empirical weight distribution in the network, the weight of a typical neighbour in our local picture will be simply the size-biased version of $W_N$, denoted $W_N^*$
\[
w(M)\stackrel{({\rm d})}{=}W_N^*.
\]

Fix $k \in [N]$, we now describe the (marked) branching process that describes the cluster exploration when started from a vertex $k$. To describe the branching process, we label the tree vertices using the standard Ulam-Harris notation, in 
particular we denote by $\emptyset$ the root of the tree, by $1$ the first offspring of the root, by $11$ the first offspring of tree vertex $1$ etc. We will write $v < w$ if $v$ comes first in the breadth-first ordering of the tree, i.e.\ 
vertices are first sorted according to length and then according to lexicographical ordering if the lengths are the same.

For the root of the branching process, we define
\[
M_{\emptyset}=k , \ X_{\emptyset} \sim \operatorname{Pois}\left(w(k)\right).
\]
Next, we define independent random variables  $\left(X_v\right)_{v \neq \emptyset}$ in two stages:
we first choose marks $\left(M_v\right)_{v \neq \emptyset}$ which are i.i.d.\ with the same distribution as  $M$.
Then, conditionally on $M_v$, let $X_v \sim \operatorname{Pois}\left(w(M_v)\right)$.
where we write  
${\rm Pois}(Y)$ for the mixed Poisson law with  random mixing parameter $Y$. 

Moreover, if we take $X_v$ to be the number of children of vertex $v$ (if it exists in the tree), this construction can be used 
to define  a (marked) Galton-Watson tree $\mathcal{T}^k$ (where only the root has a different offspring distribution).

To obtain the cluster at $k$ in $G_N$ from $\mathcal{T}^k$, we introduce a thinning procedure. 
We set $\emptyset$ to be unthinned and then explore the tree in the breadth-first order described above  and thin a tree vertex $w$ if either one of the tree vertices in the unique path between $\emptyset$ and $w$ has been thinned or if 
there exists an unthinned $v < w$ with $M_v = M_w$. 

Now, denote for $i \in [N]$, $X_v(i)$ to be the number of children of $v$ with mark $i$. 
If $v$ and $w$ are unthinned tree vertices, then we define
\begin{equation}\label{eq:edges} E(M_v, M_w) = \left\{ \begin{array}{ll} X_v(M_w) &\mbox{if } v < w ,\\
X_w(M_v) &\mbox{if } w \leq v .\end{array} \right. \end{equation}
We can  define the multigraph $\mathcal{T}^k_{\rm thin}$ by specifying that the vertex set is 
$\{ M_v \, : \, v \mbox{ unthinned}\}$ and the number of edges are given by
$( E(M_v, M_w) \, : \, v, w \mbox{ unthinned} )$.

Similarly, we can define a forest $(\mathcal{T}^1, \mathcal{T}^2, \ldots, \mathcal{T}^n)$ of independent trees constructed as above, where the root of the $k$th tree has mark $k$. Then, we can define the same thinning operation as above, starting in the tree $\mathcal{T}^1$ and going to the next when the algorithm terminates, where now also the roots of the trees may be thinned if their label has appeared in a previous tree.
If we define the edges as in~\eqref{eq:edges}, then we obtain a mulitgraph $(\mathcal{T}^1, \mathcal{T}^2, \ldots, \mathcal{T}^n)_{\rm thin}$ with vertex set $\{ M_v \, : \, v \mbox{ unthinned} \} = [N]$ and the number of edges between $i$ and $j$ given as $E(M_v, M_w)$, where $v$ and $w$ are the unique unthinned vertices $v, w$ with $M_v = i$ and $M_w =j$.

With this construction, we have the following proposition. 

\begin{proposition}\label{prop:tree_coupling}
Let $G_N^{\rm NR}$ be a realization of a Norros-Reittu mulitgraph.
For any fixed vertex $k \in [N]$, we have for the component $\sC(k)$ in $G_N^{\rm NR}$ containing $k$, 
\[ \sC(k) \stackrel{d}{=} \mathcal{T}^k_{\rm thin} . \]
Moreover, 
\[ G_N^{\rm NR} \stackrel{d}{=}(\mathcal{T}^1, \mathcal{T}^2, \ldots, \mathcal{T}^n)_{\rm thin} . \]
\end{proposition}

This proposition can be proved in the same way as Prop.\ 3.1 in \cite{norros2006conditionally}. The only difference is that we explicitly keep track of the number of edges.

\begin{figure}
\begin{center}
\includegraphics[width=13cm]{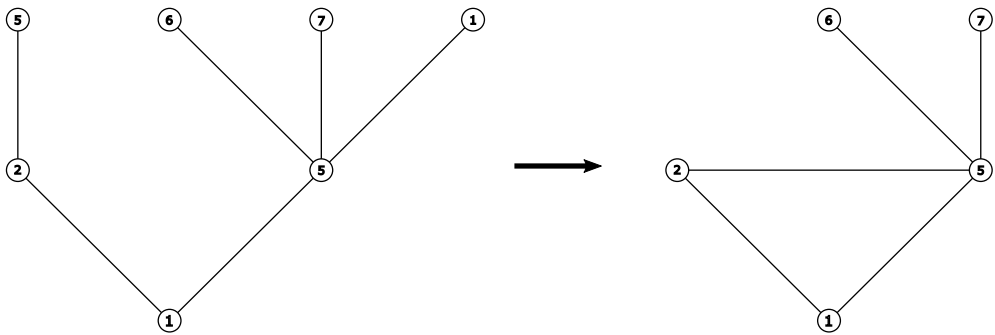}
\caption{On the left a realisation of the labelled Galton-Watson tree $\mathcal{T}^1$ and on the right the resulting  random graph. Note that the number of edges between $2$ and $5$ are determined by the number of children of type $5$ of the first child of the root.}\label{fig:thinning}
\end{center}
\end{figure}

\begin{remark}\label{rem:cycles} \emph{Thinning and creation of cycles.} 
Note that by construction,  we only create edges that lead to cycles  if there are tree vertices $v, v', w$ such that $v'$ is a child of $v$, but in the breadth-first order $v < w < v'$, $v$ and $w$ are unthinned  and such that $M_w =  M_{v'}$,
see Figure~\ref{fig:thinning} for an example.  The reason for this that the number of edges between $M_w$ and $M_v$ is determined by looking at the types of children of the first vertex in breadth-first order.
As a consequence, by this procedure we do not create edges between different components of $\mathcal{T}^i$.
Moreover, if an unthinned tree vertex $v$ has  children $v_1',\ldots,v_\ell'$ with $M_v = M_{v'_i}$, then this 
leads to $\ell$ self-loops. 
\end{remark}

\begin{remark}\label{rem:depletion}
Note that for the second construction, if the root of the $k$th tree $\mathcal{T}^k$ is not thinned, then any vertex in the tree with root $k$ that receives mark $j \leq k$ will be thinned. 
So to get a  stochastic upper bound on  the number of vertices and their degrees in the component $\sC(k)$, we can replace $\mathcal{T}^k$ by $\mathcal{T}^k_k$, where the marks are chosen independently with
distribution 
\[
M_k \stackrel{({\rm d})}{=} 
\begin{cases*}
M & if $M>k$, \\
\dagger & otherwise.
\end{cases*} \]
Then, the offspring distribution is ${\rm Pois}(W_{N,k}^*)$ with
$W_{N,k}^* \stackrel{d}{=} w(M_k)$, where we set $w(\dagger)=0$.
The error in this upper bound comes from thinning within $\mathcal{T}^k$ and also that thinned vertices are included as leaves of zero weight rather than simply being removed.
\end{remark}

\subsection{Proofs for the simple Norros-Reittu network}\label{subcrit_cpts_section}

In this section we will prove the statements made at the beginning of Section~\ref{sec:structure} using the coupling with a branching process as outlined in Section~\ref{Branching_section}.
A standard strategy will be to use that 
 the SNR model can  obtained from the MNR model by collapsing multi edges and then an  upper bound on the degrees in the  MNR model can be shown in terms of the Galton-Watson trees as described in Proposition~\ref{prop:tree_coupling}.

\begin{proof}[Proof of Proposition \ref{prop:diameter}]
By the construction of Proposition~\ref{prop:tree_coupling}, for an upper bound it suffices to 
bound the diameter in each component of $\mathcal{T}^1, \ldots, \mathcal{T}^n$ as extra edges are only created within components and so only make the diameters shorter, see also Remark~\ref{rem:cycles}.

Recall that in $\mathcal{T}^i$ the root has a ${\rm Pos}(w(i))$-distributed number of offspring, while the offspring distribution for any other vertex has the same distribution as 
$D \sim \operatorname{Pois}\left( W_N^* \right)$, where
the offspring mean satisfies
\begin{equation}\label{eq:2204-1}
\mathbb{E}\left( D \right) \rightarrow \frac{\beta}{1-2\gamma}<1.
\end{equation}
Let $(Z_k)_{k \geq 0}$ be a Galton-Watson tree with offspring distribution $D$ 
and $Z_0 = 1$. Then, for any 
$\rho \in \left(\frac{\beta}{1-2\gamma},1\right)$ and $N$ large enough we have, by Markov's inequality for all $k \geq 0$,
\[
\p(Z_k \neq 0) \leq  \mathbb{E}(Z_k)\leq \rho^k.
\]
By construction, it order to show the required bound on the diameter it suffices to bound the maximal depth of $Y$ independent Galton-Watson trees $\mathcal{T}_i^*, i =1,\ldots, Y$, with offspring distribution $D$ and where $Y$ is Poisson-distributed random variable with parameter
\begin{equation}\label{eq:total_weight}  w([N]) = \sum_{i,j \in [N]} \beta N^{2\gamma-1}i^{-\gamma}j^{-\gamma}
\sim \frac{\beta N}{(1-\gamma)^2},
\end{equation}
where here and in the following, we write $w(A) = \sum_{i \in [A]} w(i)$ for any $A \subset [N]$.

Since with high probability $Y \leq 2w([N]) \leq K N $ for a suitable constant $K \in \mathbb{N}$, we can get the required bound  by noting that for any $C >0$,
\[\begin{aligned}  \p \Big(\max_{i =1,\ldots, K N} \diam ( \mathcal{T}_i^*)  \geq C \log N \Big) 
& \leq \sum_{i=1}^{ K N} \p( \diam (\mathcal{T}_1^* ) \geq C \log N) \\ & \leq 
K N \p( Z_{\lfloor C \log N \rfloor } \neq 0 ) \leq K N \rho^{C \log N - 1} ,
\end{aligned}  \]
which converges to $0$ if we choose $C$ large enough such that $C \log \rho < 1$.
\end{proof}

In the following, we will develop a stochastic upper bound on the sizes of the trees  $\mathcal{T}^i$ in Proposition~\ref{prop:tree_coupling}
that no longer depends on $N$.
Throughout we will write $X \preceq Y$ if the random variable $Y$ stochastically dominates the random variable $X$.

\begin{lemma}\label{le:domination}
For any $\alpha > 1$, $\gamma<\frac{1}{2}$ and $N$ sufficiently large, we have
\[
 W^*_N \preceq \alpha W^* ,
\]
where $W^*$ is the weak limit of $(W_N^*)_N$ with density
\begin{equation}\label{eq:W_density}
\frac{\mathbb{P}(W^* \in {\rm d}x)}{{\rm d}x}= \mathbbm{1}_{x>\frac{\beta}{1-\gamma}} \frac{1-\gamma}{\gamma}\left( \frac{\beta}{1-\gamma}\right)^{\frac{1}{\gamma}-1} x^{-\frac{1}{\gamma}}.
\end{equation}
\end{lemma}

\begin{proof}
Note that the  MNR weights satisfy for each $i \in [N]$,
\begin{equation}\label{eq:up_bound_w}
\begin{split}
w(i)&= \beta N^{2\gamma-1} i^{-\gamma} \sum_{j = 1}^N j^{-\gamma} 
\leq\frac{\beta}{1-\gamma}\left( \frac{N}{i} \right)^\gamma =: \lambda(i) .
\end{split}
\end{equation}
Moreover, by the definition of the distribution of the marks, we have that
\[
\mathbb{P}(M \leq  k) = \sum_{\ell =1}^k \frac{\ell^{-\gamma}}{\sum_{i=1}^N i^{-\gamma}} \leq \frac{1}{1-\gamma} \frac{k^{1-\gamma}}{\sum_{i=1}^N i^{-\gamma}}. 
\]

Now we can consider  the tail of  distribution function of $W_N^* 
\stackrel{{\rm d}}{=} w(M)$ and
estimate for any $x \geq 0$,
\[
\mathbb{P}\left( W_N^* \geq x  \right)
=
\mathbb{P}\left( w(M)  \leq x  \right)
\leq \mathbb{P} (\lambda(M) \leq x )  =\mathbb{P}\left( M \leq \lambda^{-1}(x)  \right)
\leq
\frac{\left( \lambda^{-1}(x)  \right)^{1-\gamma}}{(1-\gamma)\sum_{i=1}^N i^{-\gamma}},
\]
where we write $\lambda(x) := \frac{\beta}{1-\gamma} N^\gamma x^{-\gamma}$.
Furthermore, we can compare this expression to  
\[
\mathbb{P}(\alpha W^* \geq x)
=\left( \frac{1-\gamma}{\beta}\frac{x}{\alpha}\right)^{1-1/\gamma}
=\alpha^{1/\gamma-1}\left(\frac{\lambda^{-1}(x)}{N} \right)^{1-\gamma}
\]
Therefore, we can conclude that
\[
\inf_{x\geq 0} \frac{\mathbb{P}(\alpha W^* \geq x)}{\mathbb{P}\left( W_N^* \geq x  \right)}
\geq \alpha^{1/\gamma-1} \frac{(1-\gamma)\sum_{i=1}^N i^{-\gamma}}{N^{1-\gamma}}
\rightarrow \alpha^{1/\gamma-1}>1,
\]
which gives the claimed stochastic domination.
\end{proof}

\begin{proposition}\label{le:poisson_powerlaw}
If $\alpha \in (1, \frac{1-2\gamma}{\beta})$,
then $D_\alpha \sim \operatorname{Pois}\left( \alpha W^* \right)$ satisfies
\[
p_k:=\mathbb{P}(D_\alpha=k)=\Theta(k^{-1/\gamma}).
\]
In particular, if $T$ is a Galton-Watson tree with offspring distribution $(p_k)_{k\geq 0}$, then the total size $|T|$ satisfies
\[ \mathbb{P}(|T|=k)=\Theta\big(k^{-1/\gamma} \big).\] 
\end{proposition}

\begin{proof}
The second statement follows from the first one by \cite[Thm.\ 4.1]{jonsson2011condensation} (which holds in general for trees with power law offspring distribution with an power law exponent greater than~$3$).

To prove the first statement, we use that
 by~\eqref{eq:W_density} we know that $\alpha W^*$ has density $\operatorname{f}$ where
$
\operatorname{f}(x)= C x^{-\frac{1}{\gamma}} \mathbbm{1}_{x>\frac{\alpha \beta}{1-\gamma}},
$
for some $C>0$ that makes this a probability measure.
Therefore,
\[
p_{k-1} = \frac{C}{(k-1)!}\int_{\alpha \beta/(1-\gamma)}^\infty x^{k-1-1/\gamma} e^{-x} {\rm d}x = C \frac{\Gamma(k-1/\gamma)-E_k}{\Gamma(k)}, 
\]
where the error term $E_k$ is defined as 
\[
E_k= \int_0^{\alpha \beta/(1-\gamma)} x^{k-1-1/\gamma} e^{-x} {\rm d}x
\leq \left( \frac{\alpha \beta}{1-\gamma} \right)^{k-1/\gamma}\leq 1, 
\]
and where the bound holds for $k \geq \frac{1}{\gamma} + 1$, using the assumption that $\alpha$ is small. 
We deduce that 
 $E_k=\Theta(1)$. Now we use $\Gamma (k)= \Theta \left( k^{k-{{1}/{2}}}e^{-k} \right)$ to rearrange
\[
p_{k-1}= \Theta \left( \frac{(k-1/\gamma)^{k-1/\gamma-{{1}/{2}}}e^{1/\gamma-k}}{k^{k-{{1}/{2}}}e^{-k}} \right)
= \Theta \left( \left(k-\frac{1}{\gamma}\right)^{-1/\gamma} \right)
\]
by using the classical limit
$
\left(1-\frac{1}{k\gamma}\right)^{k-1/2}  \rightarrow e^{-1/\gamma}
$.
\end{proof}

The statement of \emph{Proposition \ref{prop:big_sum_of_degrees}} follows immediately from the following  lemma, but we will also need the upper bound on the unthinned Galton-Watson forest.

\begin{lemma} \label{le:unthinned_sum_of_degrees}
For the Galton-Watson forest $(\mathcal{T}^1,\ldots, \mathcal{T}^N)$ as defined in the paragraph before Proposition~\ref{prop:tree_coupling}, we have that with high probability
\[
\max_{k \leq K_\gamma}
\frac{\sum_{z \in \mathcal{T}^k} \de(z)}{(N/k)^\gamma} \leq \log N.
\]
\end{lemma}

\begin{proof}[Proof of Lemma \ref{le:unthinned_sum_of_degrees}]
First, we consider the degrees of the roots, where we write $\emptyset_k$ for the root of
the tree $\mathcal{T}^k$. Then, since $\de(\emptyset_k) \sim 
{\rm Pois}(w(k))$, standard large deviation estimates for Poisson distributions
give a $C_1>0$ such that the event 
\[
E_1 := \Big\{ \max_{k \leq K_\gamma}
\frac{\de(\emptyset_k)}{(N/k)^\gamma}\leq C_1 \Big\}, 
\]
satisfies $\lim_{N \rightarrow \infty} \p(E_1) = 1$.	
Then, in distribution, each $\mathcal{T}^k$ consists of the root $\emptyset_k$ with 
$\de(\emptyset_K)$ edges to which we attach independent Galton-Watson trees, where the number of offspring has the same distribution as ${\rm Pois}(W_N^*)$. 
In particular, by Lemma~\ref{le:domination}, we can dominate the size of $\mathcal{T}^k$ by the size of $\mathcal{T}^{k,\alpha}$, a tree where the root $\emptyset_k$ has $\de(\emptyset_k)$ children, which are each connected to an independent Galton-Watson tree with offspring distribution $D_\alpha \sim\operatorname{Pois} ( \alpha W^*)$, for some $\alpha \in (1,\frac{1-2\gamma}{\beta}).
$ 

Now, if $T_1, \ldots, T_n$ denote independent copies of $D_\alpha$-Galton-Watson trees, then by Proposition~\ref{le:poisson_powerlaw} the total sizes $|T_i|$ of these trees satisfy $\p(|T_i| =k) = \Theta(k^{-\frac{1}{\gamma}})$, 
so that they are subexponential, see~\cite{embrechts2013modelling}, 
in the sense that
\[
\lim_{n \rightarrow \infty} \sup_{x \geq \gamma n}\left|1- \frac{\mathbb{P}\left(\sum_{k=1}^n \left(|T_k| -\mathbb{E}(|T_k|) \right)>x\right)}{n\mathbb{P}\left(|T_1|>x\right)}\right|=0.
\]
In particular, we have that for  any $\epsilon > 0$, for all $d$ sufficiently large and for $x \geq (\gamma  + \E(|T_1|))d$
\begin{equation}\label{eq:subexp}
\mathbb{P}\Big(\sum_{k=1}^d |T_k| >x  \Big) 
\leq(1+\epsilon)d \, \mathbb{P}(|T_1|>x- d \E(|T_1|) ).
\end{equation}
Since $\E(D_\alpha) =  \frac{\alpha\beta}{1-2\gamma}$, we have that
\[ \E(|T_1|) = \sum_{k=0}^\infty \Big(   \frac{\alpha\beta}{1-2\gamma}\Big)^k  = \frac{1}{1- \frac{\alpha \beta}{1- 2 \gamma}}. \]
Therefore, if we define $C_2 = 1 + C_1 \frac{\alpha\beta}{1 -2 \gamma}$, we have
for any $k \leq K_\gamma$,
\[
\begin{split}
\mathbb{P}\Big( |\mathcal{T}^k|>\Big( & \frac{N}{k} \Big)^\gamma \log N ; E_1 \Big) 
\leq \mathbb{P}\Big(|\mathcal{T}^{k,\alpha}|>\Big( \frac{N}{k} \Big)^\gamma \log N  ; E_1\Big) \\
& \leq \mathbb{P}\left( 1+ \sum_{i=1}^{\lfloor C_1 (\frac{N}{k})^\gamma\rfloor } |T_i| >\left( \frac{N}{k} \right)^\gamma \log N  \right) \\
&\leq (1+\epsilon) C_1 \left( \frac{N}{k} \right)^\gamma \mathbb{P}\left(|T_1|>\left( \frac{N}{k} \right)^\gamma \log N - C_2\left( \frac{N}{k} \right)^\gamma\right)\\
&\leq (1+2\epsilon) C_1 \left( \frac{N}{k} \right)^\gamma \mathbb{P}\left(
|T_1|>\left( \frac{N}{k} \right)^\gamma \log N 
\right),
	\end{split}
\]
where the last inequality holds for $N$ sufficiently large because we know $|T_1|$ has a power law tail. Hence, by a union bound
\[\begin{aligned} 
\mathbb{P}\bigg( \exists k \leq K_\gamma : & |\mathcal{T}^k|> \left( \frac{N}{k} \right)^\gamma \log N \bigg)\\
& \qquad \leq
(1+2\epsilon) C_1 N^\gamma \sum_{k \leq K_\gamma} k^{-\gamma}\, \mathbb{P}\left(
|T_1|>\left( \frac{N}{k} \right)^\gamma \log N\right), 
\end{aligned} \]
which for $N$ large enough we can bound for some suitable constant $C_3>0$ by
\[
\begin{split}
&
(1+2\epsilon) C_3 N^\gamma \sum_{k \leq K_\gamma} k^{-\gamma} 
\left(\left( \frac{N}{k} \right)^\gamma \log N\right)^\frac{\gamma-1}{\gamma}\\
&=
(1+2\epsilon)  C_3 N^{2\gamma-1} \log^{\frac{\gamma-1}{\gamma}}  N  \sum_{k \leq K_\gamma} k^{1-2\gamma}\\
&\leq 
\frac{(1+ 2\epsilon)  C_3}{2-2\gamma} \log^{3-\frac{1}{\gamma}-2\gamma} N = o(1),
\end{split}
\]
where we used $K_\gamma = N^{\frac{1-2\gamma}{2-2\gamma}} \log N$.
From this bound on the number of vertices, we can immediately deduce 
the claimed bound for the sum of the degrees, since
\[
\sum_{v \in \mathcal{T}^k}\de(v)=2\big|\mathcal{T}^k\big|-2 , 
\]
as each $\mathcal{T}^k$ is a tree.
\end{proof}

\begin{proposition}\label{prop:small_degrees_arent_big}
On $G_N \in \cG_{\beta,\gamma}$ with $\beta+2\gamma<1$ we have the following uniform bound on the degrees of vertices with larger index
\[
\max_{k > K_\gamma} \de(k) = O_{\mathbb{P}} \left( N^\frac{\gamma}{2-2\gamma} \right).
\]
\end{proposition}

\begin{proof}
As before we can use that the degrees are stochastically dominated by the degrees for the MNR model, where each $\de(k) \sim {\rm Pois}(w(k))$.
By  Chebyshev's inequality, we have that
\[
\begin{split}
\log \mathbb{P}\left(d(k) \geq N^\frac{\gamma}{2-2\gamma} \right) &\leq \log \mathbb{P}\left(\operatorname{Pois}( w(k) ) \geq N^\frac{\gamma}{2-2\gamma} \right)\\
&\leq \log \mathbb{E}\left(e^{\operatorname{Pois}( w(k) )} \right) - N^\frac{\gamma}{2-2\gamma} \\
&= w(k) (e-1) - N^\frac{\gamma}{2-2\gamma} .
\end{split}
\]
where by a slight abuse of notation we also write ${\rm Pois}(w(k))$ for a Poisson random variable with parameter $w(k)$.
Hence, by a union bound and using that $w(k)$  is decreasing, we obtain
\[
\begin{split}
\mathbb{P}\left( \exists k>K_\gamma : d(k)> N^\frac{\gamma}{2-2\gamma} \right)&\leq e^{- N^\frac{\gamma}{2-2\gamma}} \sum_{k=K_\gamma+1}^N e^{w(k) (e-1)} \\
&\leq (N-K_{\gamma}+1) e^{w(K_\gamma) (e-1)- N^\frac{\gamma}{2-2\gamma}} \\
&\leq N \exp \left( N^{\frac{\gamma}{2-2\gamma}} \left( -1 + \frac{\beta (e-1)}{1-\gamma} \frac{1}{\log^{\gamma} N} \right)  \right)=o(1),\\
\end{split}
\]
where in the last step we used that
$w(k) \leq \frac{\beta}{1-\gamma} (\frac{N}{k})^\gamma$, see also~\eqref{eq:up_bound_w}.
\end{proof}

\begin{proof}[Proof of Proposition \ref{prop:small_sum_of_degrees}]
As before it suffices to bound the degrees in the MNR graph.
By Remark~\ref{rem:depletion}, the tree construction yields the stochastic upper bound
\[ \max_{k \notin V_{\rm big} } \sum_{v \in \mathscr{C}(k)} \de (v)
\preceq \max_{k = K_\gamma +1,\ldots,N} \sum_{v \in \mathcal{T}^k_k} \de(v)
\leq \max_{k = K_\gamma + 1,\ldots,N} 2 |\mathcal{T}^k_k| , \]
where  $\mathcal{T}^k_k$ are independent Galton-Watson trees with the following law: 
the root of $\mathcal{T}^k_k$ has a ${\rm Pois}(w(k))$ number of offspring and all other offspring 
are independent and have a  ${\rm Pois}(W^*_{N,k})$ distribution.

	For any thinning level $z \in [N]$, we recall that $W^*_{N,z}$  is defined
as
\begin{equation}\label{truncated_weight}
W^*_{N,z}
\stackrel{({\rm d})}{=}
w (M)
\mathbbm{1}_{M>z},
\end{equation}
where $M$ is the usual mark distribution (which chooses $i \in [N]$  with probability proportional to $i^{-\gamma}$).

The same argument as in the proof of Proposition~\ref{prop:small_degrees_arent_big} also shows that there exists a constant $C_1 > 0$ such that if $\emptyset_k$ denotes the root of $\mathcal{T}_k^k$, then we have that the event
\begin{equation}\label{eq:event_E1} E_1 := \Big\{ \max_{k =K_\gamma +1 , \ldots, N} \de(\emptyset_k) \leq C_1 N^{\frac{\gamma}{2-2 \gamma}}  \Big\} \end{equation}
satisfies $\p(E_1) \rightarrow 1$ as $N \rightarrow \infty$.

To bound the size of the trees $\mathcal{T}_k^k$, we use the standard connection to random walks, see e.g.~\cite[Section 3.3]{van2016random}, where we consecutively record the number of offspring of each individual in the branching process. 

Define
\[
R:=\left\lfloor C_1 N^\frac{\gamma}{2-2\gamma} \right \rfloor, 
\]
for the same constant $C_1$ as in~\eqref{eq:event_E1}.

Then, define a random walk $(S_n)_{n\geq 0}$, where we set $S_0 = 0, S_1 = R$ and for $i \geq 1$ suppose that the increments  $S_{i+1} - S_i$
are independent and with the same distribution as $D-1$, where 
$D \sim {\rm Pois}(W^*_{N,z})$. The random walk connection then yields
that for any $z \in \{K_\gamma + 1, \ldots, N\}$ and any $L \geq 0$
\[ \p(|\mathcal{T}^z_z| > L; E_1) \leq \p(S_{L+1} \geq 0 ) . \]
Now, for large $N$
we define
\[ L := \frac{2 R}{1-\mathbb{E}(W^*_N)} , \]
which is well-defined as $\mathbb{E}(W^*_N) \rightarrow \E(W^*) < 1$.
Moreover, we define $(X_i^{(z)})_{i \geq 1}$ as a sequence of i.i.d.\ random variables with the same distribution as $D - \E(W_{N,z}^*)$, where  $D \sim {\rm Pois}(W_{N,z}^*)$. Then, we can estimate using  $\E(W_{N,z}^*) \leq \E(W_{N,z})$ 
that
 \begin{align}
\mathbb{P}\left(
|\mathcal{T}^z_z| > L; E_1\right)
& \leq \mathbb{P} \bigg( R + \sum_{i=1}^L (S_{i+1} - S_i) > L \bigg) 
\notag \\ & 
\leq
\mathbb{P} \bigg(
\sum_{i=1}^{L} X^{(z)}_i \geq L\Big(1-\mathbb{E}( W^*_{N,z} )\Big)-R
\bigg)
\notag \\ & \leq
\mathbb{P} \bigg(
\sum_{i=1}^{L} X^{(z)}_i \geq R
\bigg) . \notag
\end{align}
Then, by Markov's inequality for any $r>2 \vee \left(\frac{1}{\gamma}-1\right)$ and $N$ sufficiently large, we can deduce that 
\begin{equation}\label{post_markov}
\mathbb{P}\left(
|\mathcal{T}^z_z| > L; E_1\right)
 \leq \frac{\mathbb{E}|
\sum_{i=1}^{L}  X^{(z)}_i
|^r}{R^r}  \leq
C_2\frac{
L^{\frac{r}{2}} w(z)^{\frac{r}{2}\left( 3-\frac{1}{\gamma} \right)^+} + L w(z)^{r+1-\frac{1}{\gamma}}
}{R^r} , 
\end{equation}
where 
$C_2 > 0$ is a suitable constant, coming out of the estimate on the fractional moment, which  we defer to Lemma \ref{walk_moments}.

For  the remainder of the proof, 
we will need to fix an even  larger $r$ and assume that 
\[
r>\frac{4-4\gamma}{\gamma \wedge (1-2\gamma)}.
\]
By a union bound combined with the bound in~\eqref{post_markov} and the definitions of $L$ and $R$, we find that there exists a $C_3 >0$ such that
\begin{equation}\label{eq:4_2_union}\begin{aligned}
\p\Big( \max_{z \in \{K_\gamma+1, \ldots, N\}}  |\mathcal{T}_z^z| >L; E_1
\Big)&  \leq 
\sum_{z = K_\gamma + 1}^N
\mathbb{P}\left(
|\mathcal{T}_z^z| >L; E_1
\right)\\
& \leq
C_3 \sum_{z =K_\gamma + 1}^N
\frac{N^{\frac{r}{2}\left( 3\gamma-1 \right)^+ - \frac{r}{2}\frac{\gamma}{2-2\gamma}}}{z^{\frac{r}{2}\left( 3-\frac{1}{\gamma} \right)^+} \log^{\frac{r}{2}} N}
+
\frac{N^{\gamma r + \gamma -1 +(1-r)\frac{\gamma}{2-2\gamma}}}{z^{\gamma r + \gamma -1}\log^{r-1} N}\end{aligned} 
\end{equation}
and we require that this sum tends to $0$. For the first term, observe
\[
\frac{N^{\frac{r}{2}\left( 3\gamma-1 \right)^+ - \frac{r}{2}\frac{\gamma}{2-2\gamma}}}{z^{\frac{r}{2}\left( 3-\frac{1}{\gamma} \right)^+} }
\leq
\frac{N^{\frac{r}{2}\left( 3\gamma-1 \right)^+ - \frac{r}{2}\frac{\gamma}{2-2\gamma}}}{N^{\frac{1-2\gamma}{2-2\gamma}\frac{r}{2}\left( 3-\frac{1}{\gamma} \right)^+} }
=
\left(
N^{\frac{\left( 3\gamma-1 \right)^+-\gamma}{2-2\gamma}}
\right)^\frac{r}{2}, 
\]
and we note that the exponent of $N$ in this expression is less than $-1$ by our choice of $r$ and  since
\[
\left( 3\gamma-1 \right)^+-\gamma=
\begin{cases}
2\gamma-1,  &\mbox{if } \gamma>\frac{1}{3}, \\
-\gamma, &\mbox{if } \gamma\leq \frac{1}{3}.\\
\end{cases}
\]
In particular, the first term in~\eqref{eq:4_2_union} converges to $0$. 

For the second term, again by our choice of $r$, we have that $r > \frac{2-\gamma}{\gamma}$
so that we can deduce that
\[
\sum_{z = K_\gamma+1}^N
\frac{1}{z^{\gamma r + \gamma -1}}=O\left(
N^{\frac{1-2\gamma}{2-2\gamma}(2-\gamma r - \gamma)}
\right). 
\]
In particular, 
we have that
\[\begin{aligned}
\frac{N^{\gamma r + \gamma -1 +(1-r)\frac{\gamma}{2-2\gamma}}}{\log^{r-1} N}
\sum_{z>K_\gamma}\frac{1}{z^{\gamma r + \gamma -1}}
& =O\left(
\frac{N^{\gamma r + \gamma -1 +(1-r)\frac{\gamma}{2-2\gamma}}}{\log^{r-1} N}
N^{\frac{1-2\gamma}{2-2\gamma}(2-\gamma r - \gamma)}
\right)
\\
& =O\left(\log^{1-r} N\right)=o(1) , \end{aligned}
\]
which shows that also the second term in~\eqref{eq:4_2_union} tends to $0$. 

Thus, we can conclude from~\eqref{eq:4_2_union} that with high probability on the event $E_1$  every tree has size at most $L$. Recalling that $E_1$ occurs also with high probability and the asymptotics of $L$ then gives the required bound.
\end{proof}

The following lemma provides the moment estimate that was required in the  proof of Proposition \ref{prop:small_sum_of_degrees} above.

\begin{lemma}\label{walk_moments}
For $L \in \mathbb{N}, z \in [N]$, suppose $X^{(z)}_i, i\leq L$ are independent random variables with the same distribution as $D - \E(W_{N,z}^*)$, where $D \sim {\rm Pois} (W^*_{N,z})$ and where $W^*_{N,z}$ is defined in~\eqref{truncated_weight}. Then, 
for any $L, N \in \mathbb{N}$,  $ r>2 \vee \left(\frac{1}{\gamma}-1\right)$ and $z>1$
we have
\[
\mathbb{E}\left(
\Big|
\sum_{i=1}^{L} X^{(z)}_i
\Big|^r
\right)
\leq C \Big(
L^{\frac{r}{2}} w(z)^{\frac{r}{2}\left( 3-\frac{1}{\gamma} \right)^+} + L w(z)^{r+1-\frac{1}{\gamma}}
\Big) , 
\]
where  $C>0$ is a constant depending only on $r$ and $\gamma$
and $w(z)$ is defined in~\eqref{weight_defn}.
\end{lemma}

\begin{proof}
These calculations use a similar strategy to the proof of~\cite[Theorem 1.1]{janson2008largest}. We write $X^{(z)}$ for a random variable with the same distribution as $X^{(z)}_i$. 
We start by estimating the second and the $r$th moment of $X^{(z)}$.
First, note that as
\[
\mathbb{E} \left(
W^*_{N,z}
\right)\leq
\mathbb{E} \left(
W^*_{N}
\right)
\rightarrow
\frac{\beta}{1-2\gamma},
\]
we deduce for $N$ sufficiently large
\[ \begin{aligned}
\mathbb{E}\big( (X^{(z)})^2 \big)
& =
\operatorname{Var} \left( \operatorname{Pois} \left( X^{(z)} \right) \right)
=
\operatorname{Var} \left(
\mathbb{E} \left(
X^{(z)} | W^*_{N,z}
\right)
\right)
+
\mathbb{E} \left(
\operatorname{Var} \left(
X^{(z)} | W^*_{N,z}
\right)
\right)
\\
& =
\operatorname{Var} \left(
W^*_{N,z}
\right)
+
\mathbb{E} \left(
W^*_{N,z}
\right)
\leq
\mathbb{E} \left( \left(
W^*_{N,z}
\right)^2 \right)
+
\frac{\beta}{1-2\gamma}. 
\end{aligned} \]

Using Lemma \ref{le:domination}, we find a constant $C_1>0$ independent of $N$ for any $\alpha>1$ such that
\[
\begin{split}
\mathbb{E} \left( (W^*_{N,z})^2 \right)
&= \int_0^\infty 2x \, \mathbb{P}(W^*_{N,z}>x)\, {\rm d}x
= \left( \frac{\beta}{1-\gamma} \right)^2 + \int_{\frac{\beta}{1-\gamma}}^{w(z)} 2x \, \mathbb{P}(W^*_{N}>x)\, {\rm d} x\\
&\leq \left( \frac{\beta}{1-\gamma} \right)^2 +  \int_{\frac{\beta}{1-\gamma}}^{w(z)} 2x \, \mathbb{P}(\alpha W^*>x)\, {\rm d} x
\leq C_1 \int_{\frac{\beta}{1-\gamma}}^{w(z)} x^{2-\frac{1}{\gamma}} \, {\rm d}x\\
&\leq C_1 w(z)^{\left( 3-\frac{1}{\gamma} \right)^+},\\
\end{split}
\]
for $N$ sufficiently large, where we used the explicit density of $W^*$ identified in~\eqref{eq:W_density}.
We now have to estimate $\mathbb{E}\big((X^{(z)})^r\big)$ for  $r$ as above. We claim that
\begin{equation}\label{eq:claim_0604}
\sup_{\lambda\geq \frac{\beta}{1-\gamma}} \frac{\mathbb{E}(\operatorname{Pois}(\lambda)^r)}{\lambda^r}<\infty.
\end{equation}
Indeed,
since the Poisson distribution has an exponential moment, we know
$
\lambda \mapsto {\mathbb{E}(\operatorname{Pois}(\lambda)^r)}/{\lambda^r}
$
is finite and continuous on $\left[ \frac{\beta}{1-\gamma},\infty \right)$. Further
\[
\left\| \frac{\operatorname{Pois}(\lambda)}{\lambda} \right\|_r
\leq
\frac{ \lambda + \left\| \operatorname{Pois}(\lambda) - \lambda \right\|_r }{\lambda}
=
1+\frac{ \left\| \frac{\operatorname{Pois}(\lambda) - \lambda}{\sqrt{\lambda}} \right\|_r }{\sqrt{\lambda}}
\rightarrow
1
\]
as $\lambda \rightarrow \infty$, by the central limit theorem.
This proves the claim~\eqref{eq:claim_0604}.
Hence we can say because $X^{(z)} \preceq \operatorname{Pois} \left( W^*_{N,z} \right)$ we have some $C_2>0$ such that
\[
\mathbb{E}\left(\left(X^{(z)}\right)^r\right) \leq C_2 \mathbb{E}( (W^*_{N,z})^r)
\]
for $N$ sufficiently large. We have, given $r>\frac{1}{\gamma}-1$,
\[ \begin{aligned} 
\frac{\mathbb{E}( (W^*_{N,z})^r)}{z\mathbb{P}(M=z)w(z)^r}
& =
\frac{\sum_{j=z}^N \mathbb{P}(M=j) w(j)^r}{z\mathbb{P}(M=z)w(z)^r}
=
\sum_{j=z}^N
\frac{1}{z} \frac{j^{-\gamma} j^{-\gamma r}}{z^{-\gamma} z^{-\gamma r}}\\
& \leq z^{\gamma+\gamma r -1} \int_{z-1}^\infty j^{-\gamma-\gamma r} {\rm d} j
= \frac{1}{\gamma+\gamma r -1} \left(1-\frac{1}{z} \right)^{1-\gamma-\gamma r}\\
&\leq\frac{1}{\gamma+\gamma r -1} \left(1-\frac{1}{N} \right)^{1-\gamma-\gamma r}
=O(1), 
\end{aligned}
\]
as $N\rightarrow \infty$, noting that this bound holds uniformly over all $z>1$. Thus there is some constant $C_3$ such that, for every $z>1$,
\[
\mathbb{E}\left(\left( X^{(z)} \right)^r\right) < C_3 z\mathbb{P}(M=z)w(z)^r.
\]

Therefore by Rosenthal's inequality \cite[Chapter 3, Theorem 9.1]{gut2013probability} we obtain for $r>2\vee (\frac{1}{\gamma}-1)$
\[ \begin{aligned} 
\mathbb{E}\bigg(\Big|
\sum_{i=1}^{L} \left( X_i^{(z)} \right)
\Big|^r\bigg)
& \leq
C_3 L^{r/2} \big( \mathbb{E}(\big((X^{(z)})^2\big)\big)^{r/2}
+C_4 L  \mathbb{E}\big((X^{(z)})^r\big)\\
& \leq
C_5 L^{r/2}
\left(1+w(z)^{\left( 3-\frac{1}{\gamma} \right)^+}\right)^{r/2}
+C_6 L z\mathbb{P}(M=z)w(z)^r\\
& \leq
C_7 L^{r/2}
w(z)^{\left( 3-\frac{1}{\gamma} \right)^+\frac{r}{2}}
+C_8 L
\frac{z^{1-\gamma}w(z)^r}{N^{1-\gamma}}, 
\end{aligned} 
\]
as claimed.
\end{proof}

\begin{proof}[Proof of Proposition \ref{prop:stars_and_leaves}]
(a) Again we construct the SNR network in $\cG_{\beta,\gamma}$ via the MNR network and  use the tree construction of Proposition~\ref{prop:tree_coupling} for an upper bound. Recall by a standard concentration argument, there is some universal constant $C>0$ such that for $\lambda$ large enough
\begin{equation}\label{eq:poisson_deviations}
\mathbb{P}\left( \left| \operatorname{Pois} \left( \lambda \right) - \lambda \right| \leq \frac{\lambda}{2} \right)\leq e^{-C \lambda},
\end{equation}
so that for the unthinned degrees in $\mathcal{T}^k$ we can immediately compare $d(k)$ to $(N/k)^\gamma$. The upper bound $d(k)$ in the MNR network then follows immediately,  because thinning can only decrease the degree.

For a lower bound, it suffices to show that overall not too many vertices are being thinned in the big components. More precisely, define
\[ Z^{\rm thin} := \sum_{k =1}^{K_\gamma} \sum_{v \in \mathcal{T}^k} \1_{\{ v \small\mbox{ thinned}\}}. \]
We will show in the following $Z^{\rm thin} \leq (\log N)^3$ with high probability, which immediately implies the lower bound on the degres (as these are polynomially large) by the same concentration argument for the Poisson degrees as before.

From Lemma \ref{le:unthinned_sum_of_degrees} we have with high probability 
\[ 
\max_{k \leq K_\gamma}
\frac{|\mathcal{T}^k|}{(N/k)^\gamma}
\leq
\max_{k \leq K_\gamma}
\frac{\sum_{z \in \mathcal{T}^k} \de(z)}{(N/k)^\gamma}\leq \log N.
\]
Therefore, by summation and the definition of $K_\gamma$ we have that the event
\[
E_1 = \Big\{ 
\Big|
\bigcup_{i \leq K_\gamma} V(\mathscr{C}(i))\Big|
=
|V_{\rm Big}|
\leq
\frac{1}{1-\gamma}\sqrt{N}\log N  \Big \}, 
\]
satisfies $\p(E_1) \rightarrow 1$ as $N\rightarrow \infty$.

We can bound $Z^{\rm thin}$ by the double sum over vertices that have the same mark.  Thus, if we write $M,M'$ for two independent copies of the mark distribution, then  by distinguishing the cases of root vertices and remaining vertices, we obtain
\[ \begin{aligned}
\E(Z^{\rm thin}\1_{E_1}) & \leq 
\frac{1}{1-\gamma}
\sqrt{N} (\log N ) \mathbb{P}\left(M\leq K_\gamma \right)
+
\frac{1}{(1-\gamma)^2}
N(\log N )^{2} \mathbb{P}\left(M=M'\right) \\
& =O \left( \log^{2-\gamma} N + \log^2 N \right), 
\end{aligned}
\]
where we used that 
\begin{equation}\label{eq:MMprime}
\mathbb{P}(M=M^\prime)=\sum_{i=1}^N \left(\frac{i^{-\gamma}}{\sum_{j=1}^N j^{-\gamma}}\right)^2=\Theta\left( \frac{N^{1-2\gamma}}{N^{2-2\gamma}} \right)=\Theta\left( \frac{1}{N} \right).
\end{equation}
Hence, by Markov's inequality, we have that
\[ \p(Z^{\rm thin} \geq (\log N)^3) \leq \frac{\E(Z^{\rm thin}\1_{E_1})}{(\log N)^3} + \p(E_1^c)  , \]
and the right hand side tends to zero as $N \rightarrow \infty$.

(b) Now each of the neighbours of the root vertices $\emptyset_k$ which was not thinned has offspring $D \sim \operatorname{Pois}(W_N^*)$ and independently no children with  probability
\[
p_0=
\mathbb{P}(D=0)
=\mathbb{E}(e^{-W_N^*})
\geq
e^{-\mathbb{E}(W_N^*)}
\rightarrow
e^{-\frac{\beta}{1-2\gamma}}
>0
\]
by using Jensen's inequality. This gives the bound on $|L_k|$ by a binomial concentration argument.
\end{proof}

\begin{proof}[Proof of Lemma \ref{le:subcritical_branch_control}]
By Lemma \ref{le:unthinned_sum_of_degrees} we know that the event
\begin{equation}\label{eq:big_components}
E_1 := \bigg\{ |\mathscr{C}(k)|\leq \left(\frac{N}{k}\right)^\gamma \log N \mbox{ for all } k \leq K_\gamma \bigg\}, 
\end{equation}
satisfies $\p(E_1) \rightarrow 1$ as $N \rightarrow \infty$. 

As observed in Remark~\ref{rem:cycles}, the thinning operation does not create cycles between components, nor does it create extra edges between the root vertex and one of its children. 

We now bound the surplus of each component, which is defined as the number of edges more than edges of a tree on the same vertex set. Writing $M$ and $M'$ for two independent copies of the mark distribution, we get
\[
\mathbb{E}(\text{surplus}(\mathscr{C}(k)); E_1)\leq \mathbb{P}(M=M^\prime) \left(\frac{N}{k}\right)^{2\gamma} \log^2 N . 
\]
Hence, we obtain for the total surplus in the big components, using~\eqref{eq:MMprime}
\[
\begin{split}
\sum_{i =1}^{\left\lfloor K_\gamma \right\rfloor }\mathbb{E}(\text{surplus}(\mathscr{C}(i)); E_1)
&\leq \mathbb{P}(M=M^\prime) N^{2\gamma} \log^2 N \int_0^{K_\gamma} i^{-2\gamma} {\rm d}i \\
&=O^{\log N} \left( N^{2\gamma-1} N^{\frac{(1-2\gamma)^2}{2-2\gamma}} \right)\\
&=O^{\log N} \left( N^{\frac{2\gamma-1}{2-2\gamma}} \right)=o(1).
\end{split}
\]
Hence, combining this with the fact that $\p(E_1) \rightarrow 1$, we obtain 
by Markov's inequality that the big components form a forest with high probability.

Now that we know each component is a tree, it makes sense to talk about \emph{branches} of the root vertices. Again, we can stochastically upper bound the sizes of these branches in SNR by the ones in MNR, which are  bounded by the (unthinned) branches in the forest $\mathcal{T}^1, \ldots, \mathcal{T}^{\lfloor K_\gamma \rfloor}$.
In the latter, each of the branches is an independent ${\rm Pois}(W_N^*)$-GW tree.
Note that the total number of these trees is bounded by $\sum_{k = 1}^{\lfloor K_\gamma \rfloor} \de (\emptyset_k)$, where $\emptyset_k$ is the root of $\mathcal{T}^k$.
By the same argument as in the proof of Lemma~\ref{le:unthinned_sum_of_degrees}, we therefore have that there exists a constant $C_2 >0$ such that the event
\[ E_2 = \bigg\{ \sum_{k=1}^{\left\lfloor K_\gamma \right\rfloor }
\de(\emptyset_k) \leq C_2 \sqrt{N} \log^{1-\gamma} N  \bigg\}, \]
satisfies $\p(E_2) \ra 1$, since 
$\sum_{i=1}^{\left\lfloor K_\gamma \right\rfloor } \left(\frac{N}{i}\right)^\gamma = O(\sqrt{N} \log^{1-\gamma} N  )$.
Let $(T_i)_{i \geq 1}$ be a sequence of i.i.d.\ ${\rm Pois}(\alpha W^*)$-GW trees, where 
$\alpha \in \left(1,\frac{1-2\gamma}{\beta}\right)$. 
Further, let $J=\left\lfloor C_2 \sqrt{N} \log N \right\rfloor$.  Then, by the above argument and
Lemma~\ref{le:domination},  we have that
\[
\begin{split}
\mathbb{P}\Big( \max_{k \leq K_\gamma} \max_{B \in \mathcal{B}
(\mathfrak{C}(k))} &  \sum_{ v \in B} \de (v) \geq N^{\frac{\gamma}{2-2\gamma}}\log N ; E_2 \Big) \\
& \leq 
 \mathbb{P}\left( \max_{i=1}^{J} |T_i|>N^{\frac{\gamma}{2-2\gamma}}\log N \right)\\
 & \leq \sum_{i=1}^J \p( |T_1| > N^{\frac{\gamma}{2-2\gamma}} \log N) \\
   &=   O \Big( J \big(N^{\frac{\gamma}{2-2\gamma}} \log N \big)^{-(\frac{1}{\gamma} -1)} \Big) 
 = O \left( \log(N)^{2 - \frac{1}{\gamma} }\right) = o(1) ,
\end{split}
\]
where we used Proposition~\ref{le:poisson_powerlaw} in the  final line. 
Since $\p(E_2^c) \ra 0$, this completes the proof of the lemma.
\end{proof}

\begin{proof}[Proof of Lemma \ref{le:empirical_moment}]
Note that $\de(\emptyset_1)$, the degree of the root of $\mathcal{T}^1$,
satisfies $\de(\emptyset) \sim {\rm Pois}(w(1))$ and $w(1) = \Theta(N^\gamma)$. Hence, we can immediately deduce by standard Poisson concentration that there are constants $c_1 , C_1 > 0$ such that the event
\[ E_1 := \{ c_1 N^\gamma  \leq \de (\emptyset_1) \leq C_1 N^\gamma \} , \]
holds with high probability.

For a lower bound, we note that for the SNR model \[ \sum_{v \in \mathscr{C}(1)} \de(v)^{\eta} \geq \de(1)^{\eta_1} . \]
However, on the event $E_1$ for $\mathcal{T}_1$  the expected number of repeated labels among the children of the root $\emptyset_1$  is of order
\[
O^{\log N}\left(
N^\gamma \mathbb{P}(M=1)
+
N^{2\gamma} \mathbb{P}(M=M')
\right)
=
O^{\log N}\left(N^{2\gamma-1}\right)
=
o(1) , \]
where as before we write $M, M'$ for independent copies of the mark distribution.
So with high probability  we have that $\de(1) = \de(\emptyset_1)\geq c_1 N^\gamma$, which gives the required lower bound.

For the upper bound, by the same arguments used before we only have to bound the degrees in $\mathcal{T}^1$.
Note that with the exception of the degree of the root, all other degrees have the same distribution as $D$, where $D \sim  1+ \operatorname{Pois}\left( W_N^* \right)$.
Note also that by Lemma~\ref{le:unthinned_sum_of_degrees} there exists a constant $C_2$ such that 
the event $E_2 = \{ |\mathcal{T}^1| \leq C_2 N^{\gamma} \log N \}$ occurs with high probability. 
 Therefore, if we let $(D_i)_{i \geq 1}$ be i.i.d.\ random variables with the same distribution as $D$ and set $J := \lfloor C_2 N^{\gamma \log N} \rfloor$, then for any $\eta > 1$.
 \[ \p \bigg( \sum_{v \in \mathcal{T}^1, v \neq \emptyset_1} \de(v)^\eta 
 \geq (\log N)^3 N^{2\gamma}; E_2 \bigg) 
 \leq \p\bigg( \sum_{i=1}^J D_i^\gamma \geq (\log N)^3 N^{\eta\gamma} \bigg) .\]
 To estimate the probability on the right, we use a first moment bound. 
We notice that
\[
\mathbb{E}((W_N^*)^\eta)
= \sum_{i=1}^N \frac{w(i)^{\eta + 1} }{\sum_{j = 1}^N w(j)} = \Theta\begin{cases}
N^{\gamma(\eta +1) -1} & \eta>\frac{1}{\gamma}  - 1,\\
\log N & \eta = \frac{1}{\gamma} -1,\\
1 & \eta < \frac{1}{\gamma} - 1.
\end{cases}
\]

Moreover, there exists a constant $C_3 > 0$ such  that
\[
\mathbb{E}\left(\left(
1+ \operatorname{Pois}\left( W_N^* \right)
\right)^\eta
\right)
\leq
2^\eta \mathbb{E}\left(
1 \vee \operatorname{Pois}\left( W_N^* \right)^\eta
\right)
\leq
C_3 \mathbb{E}\left(
\left(W_N^*\right)^\eta
\right),
\]
and in particular we have that
\[
\mathbb{E} \left( \sum_{i=1}^J  D_i^\eta \right)
=O
\left( (\log N)^2
N^{(\gamma(\eta +1)-1)^+ +\gamma}
\right)
\leq
O
\left( (\log N)^2
N^{\eta\gamma}
\right) , 
\]
where we used that $\gamma < \frac 12$ in the last step. 
Hence, by Markov's inequality, 
\[ \p \bigg( \sum_{v \in \mathcal{T}^1, v \neq \emptyset_1} \de(v)^\eta 
 \geq (\log N)^3 N^{\eta\gamma}; E_2 \bigg) 
 \leq O ((\log N)^{-1} ) . 
 \] 
Since the event $E_2$ occurs with high probability and since we have by the first part that $\de(\emptyset)^\eta  \leq C_1^\eta N^{\eta\gamma}$ on the high probability event $E_1$, the upper bound in the statement of the lemma follows immediately.
\end{proof}

\begin{proof}[Proof of Proposition \ref{prop:existence_simple_double_star}]

Consider the index set
\[
I:=
\left[
N^\frac{1-2\gamma}{2-2\gamma},
N^\frac{1-2\gamma}{2-2\gamma} \log N
\right] \cap \mathbb{N}
\subset [K_\gamma].
\]

We calculate the MNR weight as
\[
w(I)\sim \sum_{k \in I} \frac{\beta}{1-\gamma}\left( \frac{N}{k} \right)^\gamma
=\Theta\left(
\sqrt{N} \log^{1-\gamma} N
\right)
\]
and so in for the MNR model the expected number of edges on the subgraph induced on $I$ is
\[
\frac{w(I)^2}{w([N])}=\Theta \left( \log^{2-2\gamma} N \right), 
\]
and so diverges to $\infty$.
Since this number is Poisson distributed, it must then be nonzero with high probability. After collapsing any multi-edges to arrive at the SNR model it must still be nonzero with high probability.

We can take any such adjacent pair $(x,y) \in I^2$ to create a double star, which by Lemma \ref{le:subcritical_branch_control} is a tree and by Proposition \ref{prop:star_degrees} has
\[
\de(x) \mbox{ and } \de(y) =\Theta_{\mathbb{P}}^{\log N} \left(
N^\frac{\gamma}{2-2\gamma}
\right).
\]

For the final claim of the Proposition we consider the empirical moment. The estimate on the moment can be proved in the same way as in the previous proof of Lemma \ref{le:empirical_moment}, but instead we now have to control the $\eta$th empirical moment
of  an i.i.d.\ sequence $D_i, i =1, \ldots, \lfloor N^\frac{\gamma}{2-2\gamma} \log N \rfloor$ where 
$D_{i} \sim 1+ \operatorname{Pois}\left( W_N^* \right).$
Then, the result follows by analogous argument, combined with  the fact that with high probability we do not see any thinning on this double star.
\end{proof}

\begin{proof}[Proof of Proposition \ref{prop:existence_long_double_star}]
We first prove the statement for the multigraph MNR.
Let  $V$ be the set of vertices with weight $w$  (as defined in~\eqref{weight_defn}) less than $1$,
\[
V:=\left\{ v \in[N] : w(v) < 1 \right\}.
\]
As in the previous proof, we consider 
$I:=
\left[
N^\frac{1-2\gamma}{2-2\gamma},
N^\frac{1-2\gamma}{2-2\gamma} \log N
\right] \cap \mathbb{N}
\subset [K_\gamma].$
We split both vertex sets into even and odd vertices as
\[
V^{\rm even}:=V \cap \left( 2\mathbb{N} \right);
\qquad
V^{\rm odd}:=V \cap \left( 2\mathbb{N} +1 \right);
\]
\[
I^{\rm even}:=I \cap \left( 2\mathbb{N} \right);
\qquad
I^{\rm odd}:=I \cap \left( 2\mathbb{N} +1 \right).
\]

Recall that $w(v) \leq \frac{\beta}{1-\gamma} (\frac{N}{v} )^\gamma$ by~\eqref{eq:up_bound_w}, so 
any vertex $v$ with $v > 
N \left( \frac{\beta}{1-\gamma} \right)^{1/\gamma}
$ is in $V$. Since we also that have by assumption $\frac{\beta}{1-\gamma}<\frac{1-2\gamma}{1-\gamma}<1$, we can conclude that
 $|V|=\Theta(N)$. Thus, 
\[
\frac{w(V^{\rm even})}{N}
\sim
\frac{w(V^{\rm odd})}{N}
\rightarrow
\rho>0, 
\]
where as before we write $w(A) = \sum_{i \in A} w(i)$ for any $A\subset [N]$.
We also recall from~\eqref{eq:total_weight} that $w([N]) \sim \frac{\beta N}{(1-\gamma)^2} =   \Theta(N)$
and finally for the large degree sets, we get
\[
w(I^{\rm even})
\sim w(I^{\rm odd})
\sim \frac{1}{2} \sum_{k \in I}\frac{\beta}{1-\gamma}\left( \frac{N}{k}\right)^\gamma
\sim \frac{\beta }{2(1-\gamma)^2} N^\gamma K_\gamma^{1-\gamma}.
\]

The number of edges from $V^{\rm even}$ to $I^{\rm even}$ in the MNR model is Poisson distributed with mean
\[
\frac{w(V^{\rm even})w(I^{\rm even})}{w([N])}
=\Theta \left(
\frac{N \cdot N^\gamma K_\gamma^{1-\gamma}}{N}
\right)
=\Theta \left(
\sqrt{N}\log^{1-\gamma}N
\right)
\]
and similarly for $V^{\rm odd}$ to $I^{\rm odd}$, so by Poisson concentration we have $\Theta_{\mathbb{P}}\left(
\sqrt{N}\log^{1-\gamma}N
\right)$ edges between each. However for any particular $v \in V$ and ${\rm par} \in \{{\rm odd}, {\rm even}\}$ we see a number of edges with mean 
\[
\frac{w(v)w(I^{\rm par})}{w([N])}
\leq
\frac{1 \cdot \frac{\beta }{2(1-\gamma)^2} N^\gamma K_\gamma^{1-\gamma}}{\frac{\beta }{(1-\gamma)^2}N}
=
\frac{\log^{1-\gamma} N}{2\sqrt{N}}
\]
so the probability that this $v$ received more than one edge from $I^{\rm par}$ is bounded by
\[
O\left(
\frac{\log^{2-2\gamma} N}{N}
\right)
\]
and hence by a union bound we will see only $O_{\mathbb{P}}^{\log N}(1)$ such instances. Further, any vertex $v \in V$ has $\de (v) \preceq \operatorname{Pois}(1)$ and so by the union bound
\[
\max_{v \in V} \de (v)=O_{\mathbb{P}}\left(
\log N
\right).
\]

So because from $\Theta_{\mathbb{P}}\left(
\sqrt{N}\log^{1-\gamma}N
\right)$ total edges, only $O_{\mathbb{P}}^{\log N}(1)$ vertices in $V^{\rm par}$ have received more than $1$, and at most $O_{\mathbb{P}}^{\log N}(1)$ edges at each, we conclude that $\Theta_{\mathbb{P}}\left(
\sqrt{N}\log^{1-\gamma}N
\right)$ vertices in $V^{\rm par}$ received a unique edge. Denote the sets which are connected by a unique edge $\cE \subset V^{\rm even}$ and $\cO \subset V^{\rm odd}$.

For the final stage of the construction, each vertex $o \in \cO$ has conditionally
\[
e(o,\cE) \succeq \operatorname{Bin}\left(
\Theta_{\mathbb{P}}\left(
\sqrt{N}\log^{1-\gamma}N
\right),
\frac{\beta}{N}
\right)
\]
so we find a single edge into $\cE$ with probability $\omega_{\mathbb{P}}\left(1/{\sqrt{N}}\right)$, and each vertex incident to this edge has no further edges with probability at least $1/e$. So, both have no further edges with probability at least $1/e^2$. Hence amongst the $|\cO|=\omega_{\mathbb{P}}\left(\sqrt{N}\right)$ trials we will find an adjacent pair each of degree $2$, with high probability.

We found a path $\mathcal{P}$ connecting
$I^{\rm odd} \leftrightarrow
V^{\rm odd} \leftrightarrow
V^{\rm even} \leftrightarrow
I^{\rm even}$ in the MNR model. 
Since each of these sets is disjoint, we know that after collapsing multi-edges to obtain the SNR model the path will still exist, and will then satisfy the criteria for our ``double star''. 
\end{proof}

\section{Voter models}\label{voter_models_section}

In this section, we will prove the two main theorems about the asymptotics of the consensus time. In Section~\ref{ssec:proof_classical}, we will consider the classical voter model and prove Theorem~\ref{class_subcrit}.  Then, in Section~\ref{ssec:proof_discursive} we will prove Theorem~\ref{obl_subcrit}
for the discursive voter model.
Throughout we will use the duality of the voter model to a system of coalescing random walks as described in Section~\ref{sec:duality}. We will also use the notation regarding various random walks statistics from that section.

\subsection{Consensus time for the classical voter model}\label{ssec:proof_classical}

In this section, we will consider the classical voter model as defined in Definition~\ref{def:voter}(a). Throughout, let $G_N \in \mathcal{G}_{\beta,\gamma}$ for $\beta +2 \gamma < 1$ be the underlying graph. 
We note that this version of the voter model fits into the general setting of a $Q$-voter model of Section~\ref{sec:duality} if for $\theta \in \mathbb{R}$ we consider $Q = Q^\theta$ defined as
\begin{equation}\label{eq:Q_theta} Q^\theta(i,j) = \de(i)^{\theta -1} \quad \mbox{if } i \sim j \mbox{ in } G_N. \end{equation}
As before, we write $\p^\theta$ for the law of (and $\E^\theta$ for the expectation with respect to)  the coalescing random walks with generator $Q^\theta$.

If we denote by $\mathscr{C}_1, \ldots, \mathscr{C}_k$ the connected components of $G_N$, then these also correspond to the irreducible components of the Markov chain with generator $Q^\theta$. So if we let $\pi = (\pi(z) , z \in V(G_N))$ be defined via
\[ \pi(z) = \frac{\de(z)^{1-\theta}}{\sum_{y \in \mathscr{C}_j} \de(y)^{1-\theta}} , \quad \mbox{ for } z \in \mathscr{C}_j, \]
for $j \in [k]$,
then $\pi|_{\mathscr{C}_j}$ is the invariant measure of the $Q^\theta$ Markov chain restricted to $\mathscr{C}_j$. 

Before the main proof, we show an elementary bound on the meeting time of two independent random walks, when the component contains a star, i.e.\ 
if there exists a vertex $k$ with a set of neighbours $L_K$, each of degree $1$ 
(compare  Proposition~\ref{prop:stars_and_leaves}).

\begin{lemma}\label{le:meeting_star}
Let $k \in [N]$ be such that $L_k$, the set of its neighbours of degree $1$, is non empty.
Let $(X_t)_{t \geq 0}$ and $(Y_t)_{t \geq 0}$ be  independent Markov chains on $\mathscr{C}(k)$ with generator $Q^\theta$.
Then, for the product chain observed on $\{ k \} \cup L_k$ (as defined before Theorem~\ref{partial_meeting}), we have
\[ t_{\rm meet}^\pi (\{ k \} \cup L_k) \leq \frac{3+\de(k)^\theta}{2} . \]
\end{lemma}

\begin{proof}
Let $S_t $ count how many of the two walkers are currently in the leaf set $L_k$. Then $(S_t)_{t \geq 0}$ is  a Markov chain on $\{0,1,2\}$ with transition rates $(s_{i j})$, where in particular
\[ s_{2 1}=2,\quad  s_{10} \geq 1 ,\quad s_{1 2} = |L_k| \de(k)^{\theta-1}\leq  \de(k)^{\theta}, \]
using that $|L_k| \leq \de(k)$.

Now, note that   $S_t = 0$ implies that  $\tau_{\rm meet} \leq t$. In particular, if $T_i = \inf\{ t \geq 0\, : \, S_t = i\}$ for $i \in \{0,1,2\}$, then we have that $T_0 \geq \tau_{\rm meet}$. 

From the explicit transition rates, we can see that $\mathbb{E}_2(T_1)=\frac{1}{2}$ and if we write $s(1) = s_{10} + s_{12}$, then
\[ \mathbb{E}_1(T_0)=\frac{1}{s(1)}+\frac{s_{12}}{s(1)}\left( \frac{1}{2}+\mathbb{E}_1(T_0) \right) \, \]
so that 
\[ \mathbb{E}_1(T_0)=\frac{1+{s_{12}}/{2}}{s_{10}}\leq 1+\frac{s_{12}}{2} . 
\]
We conclude that 
\[
\sup_{v,w \in V(\cS_k)} \E_{(v,w)} (\tau_{\rm meet})
\leq \max\{ \E_1(T_0) , \E_2(T_0) \} 
\leq \frac{1}{2}+1+\frac{s_{12}}{2}\leq \frac{3+\de(k)^\theta}{2}, 
\]
as claimed.
\end{proof}

\begin{proof}[Proof of Theorem \ref{class_subcrit}]	 We will start by showing the {\bf upper bounds}.
For the cases $\theta \geq 1$, we use that by Lemma~\ref{binary_lower_bound} and Proposition~\ref{prop:coal} we can bound 
\[\E_{\mu_u} ( \tau_{{\rm cons}} \, |\, G_{\beta,\gamma}) \leq e(2+\log N)t_{\rm hit}(G_N),\] 
where $t_{\rm hit}(G_N)  = \sup_{j \in [k]} t_{\rm hit}(\mathscr{C}_i)$ for $\mathscr{C}_1, \ldots, \mathscr{C}_k$ the components of $G_N$. Note in particular that the right hand side is still random and the expectation is only over the random walks.

We recall that the random walk associated to the classical voter model has transition rates
$Q^\theta(x,y) = \de(x)^{\theta-1}\1_{x\sim y}$ and $\pi(x) \propto \de(x)^{1-\theta}$. In particular, for any component $\sC \in {\rm Comp}(G_N)$ the conductances as defined in~\eqref{conductance_definition} are 
\[ c(xy) = \pi(x)  Q^{\theta}(x,y) = \frac{1}{\sum_{z \in \sC} \de(z)^{1- \theta}} \1_{\{ x \sim y \}} , \quad
\mbox{for any } x,y \in \sC. \]

Hence, by Proposition~\ref{prop:max_resistance}, we have for any component $\sC$,
\begin{equation}\label{eq:upper_hit} t_{{\rm hit}}(\sC) \leq \diam (\sC) \sum_{z \in \sC} \de(z)^{1- \theta} . \end{equation}

Because $\theta \geq 1$, we have that 
$\sum_{z \in \sC} \de(z)^{1-\theta} \leq |\sC|$. Therefore, by Proposition~\ref{prop:big_sum_of_degrees} and Proposition~\ref{prop:diameter}, we get that
\[ 	\sup_{\sC \in \comps} t_{{\rm hit}}(\sC)
\leq 
\max_{\sC \in \comps}\diam (\sC)
\times
\max_{\sC \in \comps}|\sC| 
=
O_{\p}^{\log N}(N^\gamma) 
, \]
which completes the upper bound for $\theta \geq 1$.

For $\theta \leq 0$, we first deal with the small components, where we  recall that the vertex set of the `small' components is defined as 
\[
V_{\rm small}:=[N]\setminus V_{\rm big},
\qquad
\mbox{where }
V_{\rm big}:= \bigcup_{k \leq K_\gamma} V \left( \mathscr{C}(k) \right)
\]
and $K_\gamma = N^\frac{1-2\gamma}{2-2\gamma}\log N$.
By Proposition~\ref{prop:small_sum_of_degrees}, we know that
\[
\max_{k \in V_{\rm small}}
\sum_{x \in \mathscr{C}(k)} \operatorname{d}(x)
=O_{\mathbb{P}}^{\log N}\left(N^{\frac{\gamma}{2-2\gamma}}\right)
\]
In particular, we get from~\eqref{eq:upper_hit} using $\sum_{i} x_i^p \leq (\sum_i x_i)^p$ for any $p\geq1$ and $x_i\geq 0$ that
\[\begin{aligned}
\max_{ k \in V_{\rm small}} t_{\rm hit}(\mathscr{C}(k)) 
& \leq \diam (G_N) \max_{ k \in V_{\rm small}} \sum_{x \in \sC(k)} \de(x)^{1-\theta}\\
 &  \leq \diam (G_N)
\max_{ k \in V_{\rm small}} \Big(\sum_{x \in \sC(k)} \de(x)\Big)^{1- \theta} 
=O_{\mathbb{P}}^{\log N}\left(N^{\frac{\gamma(1-\theta)}{2-2\gamma}}\right), 
\end{aligned} \]
where we also used Proposition~\ref{prop:diameter} to bound the diameter.

To bound the consensus time on large components, we use that by Proposition~\ref{tree_meeting_theorem}
for any $k \leq K_\gamma$,
\[ t_{\rm meet}(\sC(k)) \leq 189 \, \frac{t_{\rm hit}(k)}{\pi(k)} , \]
and find a suitable upper bound on the right hand side, which in turn gives us by Lemma~\ref{binary_lower_bound} and  Proposition~\ref{prop:coal} an upper bound on $\E_{\mu_u} ( \tau_{\rm cons}(\cC(k)) \, |\, G_N)$. 
In order to bound the invariant measure, we note that since $\theta \leq 0$, we have from Proposition \ref{prop:big_sum_of_degrees} and~\ref{prop:star_degrees} that 
\[ \min_{k \leq K_\gamma}\pi(k)  = \min_{k \leq K_\gamma} \frac{\de(k)^{1-\theta}}{\sum_{z \in \sC(k)} \de(v)^{1-\theta} } 
\geq \min_{k \leq K_\gamma} \Big( \frac{\de(k) }{ \sum_{z \in \sC(k)} \de(z) }\Big)^{1-\theta}= \Omega_{\p}^{\log N}(1) . \]

In order to bound the hitting time $t_{\rm hit}(k)$ we apply the same argument as for the small component, but 
for the random walk restricted to each  branch of $\sC(k)$ (see also Definition~\ref{branch_defn} above for the formal definition of a branch).
The bound on the sum of degrees comes from Lemma~\ref{le:subcritical_branch_control}.
Together, we obtain that
\[ \sup_{k \leq K_\gamma}  t_{\rm meet}(\sC(k)) = O_{\mathbb{P}}^{\log N}\Big(N^{\frac{\gamma(1-\theta)}{2-2\gamma}}\Big) . \]
Combined with the bound on the small components, this completes the upper bound in the case $\theta \leq 0$.

We complete the upper bounds by showing for $\theta \in (0,1)$ that
\begin{equation}\label{eq:mid_upper_bd}
\max_{k \in [N]} t_{\text{\textnormal{meet}}}\left( \mathscr{C}(k) \right)=O_{\mathbb{P}}^{\log N}\left( N^{\gamma\theta} + N^\frac{\gamma}{2-2\gamma} \right).
\end{equation}
The upper bound on the consensus time then follows by Proposition~\ref{prop:coal} and by noting that in each of the two different regimes 
one of the summands dominates.

For $k \in V_{\rm small}$, we use similar strategy as above and obtain by~\eqref{eq:upper_hit} that
\[ t_{\rm hit}(\mathscr{C}(k) ) \leq \diam(\mathscr{C}(k)) \sum_{z \in \mathscr{C}(k)} d(z)^{1-\theta} 
\leq \diam(\mathscr{C}(k)) \sum_{z \in \mathscr{C}(k)} d(z),   \]
which if we combine Proposition~\ref{prop:small_sum_of_degrees} and Proposition~\ref{prop:diameter} is seen to be
$O_{\mathbb{P}}^{\log N} \left(N^{\frac{\gamma}{2 -2 \gamma}}\right)$ uniformly in $k \in V_{\rm small}$.

For the bound on the large components, define for $k \leq K_\gamma$ the set $L_k$
as the neighbours of $k$ that have degree $1$. By Proposition~\ref{prop:leaf_counts} we have that 
 \begin{equation}\label{eq:number_star}
\min_{k \leq K_\gamma} \frac{|L_k|}{\operatorname{d}(k)}=\Omega_{\mathbb{P}}(1).
\end{equation}
hence, since all vertices in $L_k$ have degree $1$, we obtain
\[ \pi(L_k \cup \{ k\}) \geq \frac{\sum_{x \in L_k} \de(x)^{1-\theta}}{\sum_{x \in \mathscr{C}(k)} \de(x)^{1-\theta}}
\geq \frac{|L_k|}{\sum_{x \in \mathscr{C}(k)} \de(x) }
. \]

Thus by~\eqref{eq:number_star} and Proposition \ref{prop:big_sum_of_degrees} we have
\begin{equation}\label{eq:low_pi}
\min_{k \leq K_\gamma} \pi(L_k \cup\{k\}) 
=
\Omega \left( \min_{k \leq K_\gamma } \frac{\de(k)}{\sum_{x \in \mathscr{C}(k)} \de(x) } \right)
=
\Omega_{\mathbb{P}}\left( \frac{1}{\log N} \right).
\end{equation}

Because we have a large stationary mass in $L_k \cup \{k\}$, Theorem~\ref{partial_meeting} gives us that
\begin{equation}\label{eq:partial}
t_{\text{\textnormal{meet}}}(\mathscr{C}(k))
=O_{\mathbb{P}}^{\log N}\left(
t_{\rm meet}^\pi(L_k \cup \{k\})
+
 t_{\rm hit}
\left( k \right)
\right), 
\end{equation}
where we recall  that $t_{\rm meet}^\pi(L_k \cup \{k\})$ is the meeting time for the Markov chain observed on $\{k\} \cup L_k$ (see also the definition just before Theorem~\ref{partial_meeting}). We obtain from by Lemma \ref{le:meeting_star} that
\[
\max_{k \leq K_\gamma} t^\pi_{\rm meet}\left( \{k\} \cup L_k \right) \leq
\max_{k \leq K_\gamma} \frac{3+\de(k)^\theta}{2}
=O_{\mathbb{P}}(N^{\gamma \theta}) .
\]
Moreover, by Lemma \ref{le:subcritical_branch_control}
\[\begin{aligned}
\max_{k \leq K_\gamma}t_{\text{hit}}(k)
& \leq
\max_{k \leq K_\gamma} \max_{B \in \cB(\sC(k))} \operatorname{diam}(B) \sum_{v \in B} \de(v)^{1-\theta}\\
& \leq \max_{k \leq K_\gamma}
\operatorname{diam}(G_N) \max_{B \in \cB(\sC(k))} \sum_{v \in B} \de(v)
=O_{\mathbb{P}}^{\log N}\left( N^{\frac{\gamma}{2-2\gamma}} \right).
\end{aligned} 
\]
Substituting both bounds into~\eqref{eq:partial}, we obtain 
\[ \max_{k \leq K_\gamma} t_{\rm meet}(\mathscr{C}(k) )
= O_{\mathbb{P}}^{\log N} \left(N^{\gamma \theta}+N^{\frac{\gamma}{2-2\gamma}} \right). \]
By combining this with the bound on the small components, we have completed the proofs for the upper bounds in all cases.

We continue with the {\bf lower bounds}.

For the first part, we suppose that $\theta >0$ and consider the consensus time on $\sC(1)$. By Lemma~\ref{binary_lower_bound} and Proposition \ref{prop:coal}
\[ \mathbb{E}^{\theta}_{\mu_u}(\tau_{\text{\textnormal{cons}}}(\mathscr{C}(1)) \, |\, G_N)
\geq 2u(1-u)t_{\text{meet}}(\mathscr{C}(1))
\geq 2u(1-u)t_{\text{meet}}^\pi(\mathscr{C}(1)), \]
where the last inequality follows from the definitions. To bound the right hand side, we recall from Proposition~\ref{lower_meeting_bound} 
that 
\begin{equation}\label{eq:L2} t_{\text{meet}}^\pi(\mathscr{C}(1)) \geq \frac{( 1- \sum_{x \in \mathscr{C}(1)} \pi(x)^2)^2}{4
\sum_{x \in \mathscr{C}(1)} q(x) \pi(x)^2 } . \end{equation}
In order to find a lower bound on the right hand side, we first bound the maximum of the invariant distribution. 
If $\theta \in (0,1)$, then we have by Proposition~\ref{prop:star_degrees}
\[\max_{v \in \mathscr{C}(1)} \pi(v) 
 = \frac{\max_{v \in \mathscr{C}(1)} \de(v)^{1 - \theta}}{ \sum_{z \in \mathscr{C}(1)} 
\de(z)^{1-\theta}} \leq \frac{ |\mathscr{C}(1)|^{1 - \theta}}{ |\mathscr{C}(1)|} \leq \de(1)^{-\theta}
= O_{\mathbb{P}} (N^{-\gamma \theta} ) . \]
Similarly, if $\theta \geq 1$ we recall the leaf neighbours of Proposition~\ref{prop:leaf_counts}, 
\[ \max_{v \in \mathscr{C}(1)} \pi(v) 
\leq \frac{1}{ \sum_{z \in \mathscr{C}(1)} 
\de(z)^{1-\theta}}
\leq \frac{1}{ |L_1| }
= O_{\mathbb{P}} (N^{-\gamma})  . \]

In particular, in both cases we have
\[ \sum_{x \in \mathscr{C}(1)}\pi(x)^2 \leq \max_{x \in \mathscr{C}(1)} \pi(x) \sum_{v \in \mathscr{C}(1)} \pi(v) = o_{\mathbb{P}}(1). \]

To estimate the denominator in~\eqref{eq:L2}, we note that 
for $\theta \geq 1$,
\[\begin{aligned} \sum_{v \in \mathscr{C}(1)}  q(v) \pi(v)^2 & =\frac{\sum_{v \in \mathscr{C}(1)}\operatorname{d}(v)^{\theta}\operatorname{d}(v)^{2-2\theta}}{\left( \sum_{v \in \mathscr{C}(1)}\operatorname{d}(v)^{1-\theta}  \right)^2}
\leq \frac{\sum_{v \in \mathscr{C}(1)}\operatorname{d}(v) }{|L_1|^2}\\
& = O_{\mathbb{P}}^{\log N} \Big( \frac{N^\gamma}{N^{2 \gamma} } \Big) 
= O_{\mathbb{P}}^{\log N} (N^{-\gamma}), \end{aligned} 
\]
where we used Proposition~\ref{prop:leaf_counts} for the denominator and Proposition~\ref{prop:big_sum_of_degrees} for the numerator.
By the same results and Lemma \ref{le:empirical_moment}, we have for
 $\theta \in (0,1)$, 
\[  \sum_{v \in \mathscr{C}(1)}  q(v) \pi(v)^2 \leq \frac{ \sum_{v \in \mathscr{C}(1)}\operatorname{d}(v)^{2- \theta} }{\big( \sum_{v \in \mathscr{C}(1)}\operatorname{d}(v)^{1-\theta}  \big)^2} = O^{\log N}_{\mathbb{P}} \left( \frac{N^{(2-\theta)\gamma}}{N^{2 \gamma} } \right) 
= O^{\log N}_{\mathbb{P}} ( N^{- \theta \gamma}) . 
\] 
Hence, we obtain from~\eqref{eq:L2}  for $\theta > 0$
\begin{equation}\label{eq:low_bd_1} t_{\text{meet}}^\pi(\mathscr{C}(1)) = \left\{ \begin{array}{ll} \Omega^{\log N}_{\mathbb{P}} ( N^\gamma) & \mbox{if } \theta \geq 1 ,\\ \Omega^{\log N}_{\mathbb{P}} (N^{\gamma \theta}) &\mbox{if } \theta \in (0,1). \end{array} \right.  \end{equation}

For the second of the part of the lower bound, we use a component that contains a sufficiently large ``double star'' structure and consider parameters $\theta<1$. More precisely, 
by Proposition~\ref{prop:existence_simple_double_star}, with high probability, there exists a tree component that contains two adjacent vertices $x$ and $y$ such that
\begin{equation}\label{simple_double_star_properties}
 \de(x),\de(y)
\mbox{ and }\sum_{v \in \sC(x)} \de(v) 
\mbox{ are }
\Theta_{\mathbb{P}}^{\log N}\left( N^\frac{\gamma}{2-2\gamma} \right).
\end{equation}

Now, let $A_x$ be the set of vertices in $\sC(x)$ that are closer to $x$ than to $y$, and $A_y$ the complement.
Then, we will use that by Proposition~\ref{conductance_theorem}
\begin{equation}\label{eq:cond_bound} \E_{\mu_u}(\tau_{\rm cons} \, |\, G_N) = \Omega \left( \frac{\pi(A_x)\pi(A_y)}{ \sum_{v \in A_x} \sum_{w \in A_y} c(vw) }  \right) . \end{equation}

We start by estimating the term $\pi(A_x)\pi(A_y)$.
Note that for $\theta \in (0,1)$, we have that
\[ \de(x) \leq |A_x| \leq \sum_{v \in A_x} \de(v)^{1-\theta} 
\leq \sum_{v \in  \mathscr{C}(x) } \de(v)^{1-\theta} \leq 
\sum_{v \in \mathscr{C}(x)} \de (v)  , \]
and the same bounds hold when replacing $x$ by $y$. 
Therefore, for $\theta \in (0,1)$, we obtain
\[
\frac{\de(x)}{\sum_{v \in \sC(x)} \de (v)}
\leq
\frac{\sum_{v \in A_x} \de (v)^{1-\theta}}{\sum_{v \in A_y} \de (v)^{1-\theta}}
\leq
\frac{\sum_{v \in \sC(x)} \de (v)}{\de(y)}, 
\]
so that we can deduce from~\eqref{simple_double_star_properties} that 
\[ \pi(A_x) \pi(A_y)
=\left(
\sqrt{\frac{\sum_{v \in A_x} \de (v)^{1-\theta}}{\sum_{v \in A_y} \de (v)^{1-\theta}}}+\sqrt{\frac{\sum_{v \in A_y} \de (v)^{1-\theta}}{\sum_{v \in A_x} \de (v)^{1-\theta}}}
\right)^{-2}
= \Om^{\log N}_{\mathbb{P}} (1) . \]
Furthermore, since $\sC(x)$ is a tree the denominator in~\eqref{eq:cond_bound} reduces to
\[ 
c(xy) = \frac{1}{\sum_{v \in \sC(x)} \de(v)^{1-\theta} }
\leq
\frac{1}{\left| \sC(x) \right|}=
O_{\mathbb{P}}^{\log N} \left( N^{- \frac{\gamma}{2-2\gamma}} \right)
.
\]

We finally consider the case $\theta \leq 0$ on this double star. 
Again, we start by estimating $\pi(A_x) \pi(A_y)$. By Proposition \ref{prop:existence_simple_double_star} we have 
\[
\sum_{v \in A_x}\de(v)^{1-\theta}
=O_{\mathbb{P}}\Big( N^\frac{\gamma(1-\theta)}{2-2\gamma}\Big).
\]
Since further $\de(x)=\Theta_{\mathbb{P}}(N^\frac{\gamma}{2-2\gamma})$ by~\eqref{simple_double_star_properties}, 
we also have that
\[ \sum_{v \in A_x}\de(v)^{1-\theta}
 \geq \de(x)^{1-\theta} = \Omega_{\mathbb{P}}\Big( N^\frac{\gamma(1-\theta)}{2-2\gamma}\Big).
\] 
The same bounds hold for $y$ and so by the same argument as above, 
we have that $\pi(A_x)\pi(A_y)=\Omega_{\mathbb{P}}(1)$. Moreover,
\[
c(xy)=\frac{1}{\sum_{v \in \sC(x)} \de(v)^{1-\theta} }
\leq \frac{1}{\de(x)^{1-\theta}}
= O_{\mathbb{P}}^{\log N} \Big( N^\frac{-\gamma(1-\theta)}{2-2\gamma} \Big).
\]
Combining the estimates on the stationary distribution and the conductance $c(xy)$, 
we conclude from \eqref{eq:cond_bound} that
 \begin{equation}\label{eq:low_bd_2}
 \E_{\mu_u}(\tau_{\rm cons} \, |\, G_N) = \left\{ \begin{array}{ll} \Om_{\mathbb{P}}^{\log N} ( N^\frac{\gamma}{2-2\gamma} ) 
 & \mbox{if } \theta \in (0,1), \\
 \Om_{\mathbb{P}}^{\log N} ( N^\frac{\gamma(1-\theta)}{2-2\gamma} ) 
 & \mbox{if } \theta \in (-\infty,1]. \end{array} \right.
 \end{equation}

Combining Equations~\eqref{eq:low_bd_2} and~\eqref{eq:low_bd_1} completes the proof of Theorem~\ref{class_subcrit} by giving all the required lower bounds.
 \end{proof}

\subsection{Consensus time for the discursive voter model}\label{ssec:proof_discursive}

In this section, we will consider the discursive voter model as defined in Definition~\ref{def:voter}(b). This version of the voter model fits into the general setting of a $Q$-voter model of Section~\ref{sec:duality} if for $\theta \in \mathbb{R}$ we consider $Q = \mathbf{Q}^\theta$ defined as
\begin{equation}\label{eq:Q_theta_discursive} \mathbf{Q}^\theta(i,j) = 
\frac{\operatorname{d}(i)^{\theta-1}+\operatorname{d}(j)^{\theta-1}}{2}
 \quad \mbox{if } i \sim j \mbox{ in } G_N. \end{equation}
As before, we write $\mathbf{P}^\theta$ for the law of (and $\mathbf{E}^\theta$ for the expectation with respect to)  the coalescing random walks with generator $\mathbf{Q}^\theta$.

If we denote by $\mathscr{C}_1, \ldots, \mathscr{C}_k$ the connected components of $G_N$, then define $\pi = (\pi(z) , z \in V(G_N))$ via
\[ \pi(z) = \frac{1}{|\mathscr{C}_j|} , \quad \mbox{ for } z \in \mathscr{C}_j, \]
for $j \in [k]$.
Then $\pi|_{\mathscr{C}_j}$, i.e.\ the uniform measure on $\mathscr{C}_j$, is the invariant measure of the $\mathbf{Q}^\theta$ Markov chain restricted to $\mathscr{C}_j$.
 
First we require another application of Theorem~\ref{partial_meeting}, which is simpler than for the classical voter model,  but covers a wider range of cases.
 
\begin{lemma}\label{meeting_obl_subcrit} For $G_N \in \mathcal{G}_{\beta,\gamma}$ with $\beta + 2\gamma <1$, we have that
\[
t_{\text{\textnormal{meet}}}(G_N)  = \sup_{j \in [k]} t_{\rm meet}(C_i) =
\begin{cases}
O^{\log N}_{\mathbb{P}}\left( N^\frac{\gamma}{2-2\gamma} \right) &  \theta > \frac{3-4\gamma}{2-2\gamma} \\
O^{\log N}_{\mathbb{P}}\left( N^{\gamma(2-\theta)} \right) &  1 < \theta \leq \frac{3-4\gamma}{2-2\gamma} \\
O^{\log N}_{\mathbb{P}}\left( N^\gamma \right) & 2\gamma \leq \theta \leq 1\\
O^{\log N}_{\mathbb{P}}\left( N^\frac{\gamma(2-\theta)}{2-2\gamma} \right) & \theta < 2\gamma
\end{cases}
\]
\end{lemma}

\begin{proof}
By Proposition~\ref{le:subcritical_branch_control}, we can
work on the high probability set where all big components are  trees. 
Recall that $ K_\gamma:=N^\frac{1-2\gamma}{2-2\gamma} \log N$
and denote for any $k \leq K_\gamma$
by $L_k$ the set of degree $1$ vertices adjacent to $k$. By Proposition~\ref{prop:leaf_counts} and since the stationary distribution is  uniform
\[
\min_{k \leq K_\gamma} \pi(L_k)= \min_{k \leq K_\gamma} \frac{|L_k|}{|\mathscr{C}(k)|}=\Omega^{\log N}_{\mathbb{P}}(1).
\]

Then by exchangeability, coalescence observed in $L_k$ is just complete graph (Wright-Fisher) coalescence. This is because a simultaneous move by both walkers gives the same probability to coalesce as a single move. Thus, coalescence occurs for the partially observed process at rate
\[
\frac{1+\operatorname{d}(k)^{\theta-1}}{|L_k|}
\]
and we conclude by Proposition~\ref{prop:leaf_counts}
\begin{equation}\label{eq:2703-1}
\max_{k \leq K_\gamma}  t^\pi_{\text{\textnormal{meet}}}( L_k )   
\leq \max_{k \leq K_\gamma} \frac{|L_k|}{1+\operatorname{d}(k)^{\theta-1}}
= 
\begin{cases}
O_{\mathbb{P}}^{\log N}(N^{\gamma(2-\theta)}) & \theta>1, \\
O_{\mathbb{P}}^{\log N}(N^\gamma) & \theta \leq 1 .
\end{cases}
\end{equation}

Now, we let $\mathcal{S}$ be the collection of small components and branches in large components. If we denote by $P_{x,y}$ the set of paths between any vertices $x$ and $y$, then by Proposition~\ref{prop:max_resistance} we obtain
\[
\begin{split}
\max_{S \in \mathcal{S}} t_{\text{hit}}(S) &\leq \max_{S \in\mathcal{S}} \max_{x,y \in V(S)} \min_{P_{x,y}} \sum_{\{u,v\} \in E(P_{x,y})} \frac{2|S|}{\operatorname{d}(u)^{\theta-1}+\operatorname{d}(v)^{\theta-1}}\\
&\leq \max_{S \in\mathcal{S}} |S| \operatorname{diam}(S) \max_{v \in S} \left( \operatorname{d}(v)^{1-\theta} \right)\\
&\leq \max_{S \in\mathcal{S}} |S| \operatorname{diam}(G_{\beta,\gamma}) \max_{v > K_\gamma} \left( \operatorname{d}(v)^{1-\theta} \right)\\
&=
\begin{cases}
 O^{\log N}_{\mathbb{P}}\left(N^{\frac{(2-\theta)\gamma}{2-2\gamma}}\right) & \theta<1 ,\\
 O^{\log N}_{\mathbb{P}}\left(N^{\frac{\gamma}{2-2\gamma}}\right) & \theta \geq 1,
\end{cases}
\end{split}
\]
where we used Lemma~\ref{le:subcritical_branch_control} and Propositions~\ref{prop:small_sum_of_degrees} and \ref{prop:diameter} in the last step.

If we combine this last bound with~\eqref{eq:2703-1} and apply
Theorem~\ref{partial_meeting}, then we obtain
\[
\begin{split}
\max_{k \leq K_\gamma}t_{\text{\textnormal{meet}}}(\mathscr{C}(k))
&=
\begin{cases}
O^{\log N}_{\mathbb{P}}\left( N^{\gamma(2-\theta)}+ N^\frac{\gamma}{2-2\gamma} \right) &  \theta > 1 ,\\
O^{\log N}_{\mathbb{P}}\left( N^\gamma + N^\frac{\gamma(2-\theta)}{2-2\gamma} \right) & \theta \leq 1,
\end{cases}\\
&=
\begin{cases}
O^{\log N}_{\mathbb{P}}\left( N^\frac{\gamma}{2-2\gamma} \right) &  \theta > \frac{3-4\gamma}{2-2\gamma} ,\\
O^{\log N}_{\mathbb{P}}\left( N^{\gamma(2-\theta)} \right) &  1 < \theta \leq \frac{3-4\gamma}{2-2\gamma} ,\\
O^{\log N}_{\mathbb{P}}\left( N^\gamma \right) & 2\gamma \leq \theta \leq 1,\\
O^{\log N}_{\mathbb{P}}\left( N^\frac{\gamma(2-\theta)}{2-2\gamma} \right) & \theta < 2\gamma,
\end{cases}
\end{split}
\]
which completes the proof of the lemma.
\end{proof}

\begin{proof}[Proof of Theorem \ref{obl_subcrit}]

The upper bound for all four cases follows immediately from 
Lemma~\ref{binary_lower_bound}, Propositon \ref{prop:coal} and Lemma~\ref{meeting_obl_subcrit}, so it only remains to prove the lower bounds. For these, it will be very useful that {the stationary distribution $\pi$ on each component  is always uniform.}

The first lower bound is for the case $\theta \geq \frac{3-4\gamma}{2-2\gamma}$ for which we must consider the \emph{long} double star component, whose existence is proved Proposition \ref{prop:existence_long_double_star}. First, we note that the `separating edge' $\{ v_2,v_3\}$ on the
the long double star has conductance
\[
c(v_2v_3)=O_{\mathbb{P}}^{\log N} \left( N^{-\frac{\gamma}{2-2\gamma}} \right). 
\]
Moreover, by Propositions \ref{prop:big_sum_of_degrees} and \ref{prop:star_degrees}
\[
\de(v_1),\de(v_4) \mbox{ and } \left| \sC( v_1 ) \right| \mbox{ are }
\Theta_{\mathbb{P}}^{\log N}\left( N^\frac{\gamma}{2-2\gamma} \right)
\]
which implies that we have $\Theta_{\mathbb{P}}^{\log N}(1)$ stationary mass on each side (by a similar argument as before). Hence by Proposition \ref{conductance_theorem} we have consensus time $\Omega_{\mathbb{P}}^{\log N} \left( N^{\frac{\gamma}{2-2\gamma}} \right).$

For the lower bound when $2\gamma<\theta < \frac{3-4\gamma}{2-2\gamma}$, we apply Corollary \ref{lower_meeting_bound} to $\mathscr{C}(1)$ to see that
\[
t_{\text{\textnormal{meet}}}^{\pi}( \mathscr{C}(1))
\geq \frac{(1-\sum_{v \in \mathscr{C}(1)} \pi(v)^2)^2}{4\sum_{v \in \mathscr{C}(1)} q(v) \pi(v)^2}
=\Theta_{\mathbb{P}}\left( \frac{|\mathscr{C}(1)|^2}{\sum_{v \in \mathscr{C}(1)} q(v)} \right)
=\Theta_{\mathbb{P}}\left(
 \frac{|\mathscr{C}(1)|^2}{\sum_{v \in \mathscr{C}(1)} \operatorname{d}(v)^\theta}\right).
\]
Recall the moment calculation in Lemma \ref{le:empirical_moment} to see that when $\theta \geq 1$
\[
\sum_{v \in \sC(1)} \de(v)^\theta
=
\Theta^{\log N}_{\mathbb{P}}\left( N^{\gamma\theta} \right), 
\]
whereas for $\theta \in (2\gamma,1)$ we instead have by Proposition \ref{prop:big_sum_of_degrees}
\[
\sum_{v \in \sC(1)} \de(v)^\theta \leq \sum_{v \in \sC(1)} \de(v)=O^{\log N}_{\mathbb{P}}(N^\gamma ).
\]
Combining the statements yields
\[
\frac{|\mathscr{C}(1)|^2}{\sum_{v \in \mathscr{C}(1)} \operatorname{d}(v)^\theta}=\Omega^{\log N}_{\mathbb{P}}\left(
N^{(2-\theta)\gamma}\vee N^\gamma
\right) .
\]
By Lemma \ref{binary_lower_bound} and Proposition \ref{prop:coal} this expression gives a lower bound for the consensus time.

For the final case, when $\theta <2\gamma$, we require another double star component, but this one must be that without a path, whose existence is stated in Proposition \ref{prop:existence_simple_double_star}. This double star is a tree structure with two adjacent ``star'' vertices $x$, $y$, where
\[
\de(x),\de(y)\mbox{ and } \left| \sC( x ) \right| \mbox{ are }
\Theta_{\mathbb{P}}^{\log N}\left( N^\frac{\gamma}{2-2\gamma} \right) . 
\]
Therefore, we have stationary mass of $\Theta_{\mathbb{P}}^{\log N}(1)$ in the vertices closest to $x$ and in those closest to $y$.
We note that
\[
\mathbf{Q}^\theta (x,y)=O_{\mathbb{P}}^{\log N}\left( N^\frac{\gamma(\theta-1)}{2-2\gamma} \right)
\]
and since the stationary distribution is uniform
\[
\pi(x)=O_{\mathbb{P}}^{\log N}\left( N^{-\frac{\gamma}{2-2\gamma}} \right).
\]
Thus, by the definition of the conductance~\eqref{conductance_definition}
\[
c(xy)=\pi(x)\mathbf{Q}^\theta (x,y)
=O_{\mathbb{P}}^{\log N}\left( N^\frac{\gamma(\theta-1)}{2-2\gamma} N^{-\frac{\gamma}{2-2\gamma}} \right).
\]
Combining the estimate on the stationary mass and the conductance, we have by Proposition \ref{conductance_theorem},
\[
t_{\rm meet}( \mathscr{C}(1)) 
=\Omega_{\mathbb{P}}^{\log N}\left(
N^{\frac{\gamma(2-\theta)}{2-2\gamma}}
\right),
\]
which gives the last remaining lower bound.
\end{proof}

{\bf Acknowledgements.}
We would like to thank Peter M\"orters 
and Alexandre Stauffer  
for many useful discussions.
JF is supported by a scholarship from the EPSRC Centre for Doctoral Training in Statistical Applied Mathematics at Bath (SAMBa), under the project EP/L015684/1.

\bibliographystyle{abbrv}
\bibliography{mathsci_bib}

\end{document}